\documentclass[graybox,envcountsame,envcountresetchap]{svmult}

\usepackage{mathptmx}       
\usepackage{helvet}         
\usepackage{courier}        
\usepackage{type1cm}        
\usepackage{makeidx}         
\usepackage{graphicx}        
\usepackage{multicol}        
\usepackage[bottom]{footmisc}

\usepackage{amssymb}
\usepackage{amsmath}
\usepackage{amsfonts}
\usepackage{comment}
\usepackage{graphicx}
\usepackage{tikz,ifthen}
\usetikzlibrary{calc}
\usepackage{mathtools}
\usepackage{etoolbox}
\usetikzlibrary{decorations.markings,patterns,decorations.pathreplacing,intersections}
\usetikzlibrary{backgrounds}

\DeclareMathOperator{\SL}{SL}
\DeclareMathOperator{\Sp}{Sp}
\DeclareMathOperator{\GL}{GL}
\DeclareMathOperator{\Arg}{Arg}

\newcommand{\fp}{\qed}

\newcommand{\psmm}[4]{\left(\begin{smallmatrix}{#1}&{#2}\\{#3}&{#4}\end{smallmatrix}\right)}
\renewcommand{\H}{\mathcal H}
\newcommand{\FF}{\mathfrak F}

\renewcommand{\E}{\mathcal E}
\newcommand{\RC}{\mathcal R}
\newcommand{\ov}[1]{\overline{#1}}
\newcommand{\wh}[1]{\widehat{#1}}
\newcommand{\Q}{{\mathbb Q}}
\newcommand{\Z}{{\mathbb Z}}
\renewcommand{\P}{{\mathbb P}}
\newcommand{\F}{{\mathbb F}}
\newcommand{\R}{{\mathbb R}}
\newcommand{\C}{{\mathbb C}}
\newcommand{\Om}{\Omega}

\newcommand{\G}{\Gamma}

\newcommand{\om}{\omega}

\newcommand{\al}{\alpha}
\newcommand{\be}{\beta}
\newcommand{\ga}{\gamma}

\newcommand{\la}{\lambda}
\renewcommand{\th}{\theta}

\newcommand{\z}{\zeta}
\newcommand{\eps}{\varepsilon}
\renewcommand{\pmod}[1]{\allowbreak\ ({\rm{mod}}\,\,#1)}
\newcommand{\lgs}[2]{\mbox{$\left(\frac{#1}{#2}\right)$}}
\newcommand{\leg}[2]{\mbox{$\left(\dfrac{#1}{#2}\right)$}}
\newcommand{\isom}{\simeq}
\DeclareMathOperator{\Tr}{Tr}

\def\cal{\mathcal}

\makeatletter
\def\renewtheorem#1{%
  \expandafter\let\csname#1\endcsname\relax
  \expandafter\let\csname c@#1\endcsname\relax
  \gdef\renewtheorem@envname{#1}
  \renewtheorem@secpar
}
\def\renewtheorem@secpar{\@ifnextchar[{\renewtheorem@numberedlike}{\renewtheorem@nonumberedlike}}
  \def\renewtheorem@numberedlike[#1]#2{\newtheorem{\renewtheorem@envname}[#1]{#2}}
  \def\renewtheorem@nonumberedlike#1{
    \def\renewtheorem@caption{#1}
    \edef\renewtheorem@nowithin{\noexpand\newtheorem{\renewtheorem@envname}{\renewtheorem@caption}}
    \renewtheorem@thirdpar
  }
  \def\renewtheorem@thirdpar{\@ifnextchar[{\renewtheorem@within}{\renewtheorem@nowithin}}
    \def\renewtheorem@within[#1]{\renewtheorem@nowithin[#1]}
    \makeatother

\renewtheorem{theorem}{Theorem}[section]
\renewtheorem{remarks}[theorem]{{\it Remarks}}
\smartqed

\makeindex             

\begin{document}

\title*{An Introduction to Modular Forms}
\author{Henri Cohen}
\institute{Henri Cohen \at Institut de Math\'ematiques de Bordeaux, Universit\'e de Bordeaux, 351 Cours de la Lib\'eration, 33405 TALENCE Cedex, FRANCE, \email{Henri.Cohen@math.u-bordeaux.fr}}

%
%
\maketitle

\abstract{\\ In this course we introduce the main notions relative to the
classical theory of modular forms. A complete treatise in a similar style
can be found in the author's book joint with F.~Str\"omberg \cite{Coh-Str}.}

\section{Functional Equations}

Let $f$ be a complex function defined over some subset $D$ of $\C$.
A \emph{functional equation} is some type of equation relating the value
of $f$ at any point $z\in D$ to some other point, for instance
$f(z+1)=f(z)$. If $\ga$ is some function from $D$ to itself, one can ask
more generally that $f(\ga(z))=f(z)$ for all $z\in D$ (or even
$f(\ga(z))=v(\ga,z)f(z)$ for some known function $v$). It is clear that
$f(\ga^m(z))=f(z)$ for all $m\ge0$, and even for all $m\in\Z$ if $\ga$
is invertible, and more generally the set of bijective functions $u$ such that
$f(u(z))=f(z)$ forms a \emph{group}.

Thus, the basic setting of functional equations (at least of the type that we
consider) is that we have a group of transformations $G$ of $D$,
that we ask that $f(u(z))=f(z)$ (or more generally $f(u(z))=j(u,z)f(z)$ for
some known $j$) for all $u\in G$ and $z\in D$, and we ask for some type of
regularity condition on $f$ such as continuity, meromorphy, or holomorphy.

Note that there is a trivial but essential way to construct from scratch
functions $f$ satisfying a functional equation of the above type: simply
choose any function $g$ and set $f(z)=\sum_{v\in G}g(v(z))$. Since $G$ is a
group, it is clear that \emph{formally} $f(u(z))=f(z)$ for $u\in G$. Of course
there are convergence questions to be dealt with, but this is a fundamental
construction, which we call \emph{averaging} over the group.

We consider a few fundamental examples.

\subsection{Fourier Series}

We choose $D=\R$ and $G=\Z$ acting on $\R$ by translations. Thus, we ask
that $f(x+1)=f(x)$ for all $x\in\R$. It is well-known that this leads to the
theory of \emph{Fourier series}: if $f$ satisfies suitable regularity
conditions (we need not specify them here since in the context of modular forms
they will be satisfied) then $f$ has an expansion of the type
$$f(x)=\sum_{n\in\Z}a(n)e^{2\pi inx}\;,$$
absolutely convergent for all $x\in\R$, where the \emph{Fourier coefficients}
$a(n)$ are given by the formula
$$a(n)=\int_0^1e^{-2\pi inx}f(x)\,dx\;,$$
which follows immediately from the orthonormality of the functions
$e^{2\pi imx}$ (you may of course replace the integral from $0$ to $1$
by an integral from $z$ to $z+1$ for any $z\in\R$).

An important consequence of this, easily proved, is the \emph{Poisson
summation formula}: define the \emph{Fourier transform} of $f$ by
$$\wh{f}(x)=\int_{-\infty}^{\infty}e^{-2\pi i xt}f(t)\,dt\;.$$
We ignore all convergence questions, although of course they must be
taken into account in any computation.

Consider the function $g(x)=\sum_{n\in\Z}f(x+n)$, which is exactly the
averaging procedure mentioned above. Thus $g(x+1)=g(x)$, so $g$ has
a Fourier series, and an easy computation shows the following
(again omitting any convergence or regularity assumptions):

\begin{proposition}[Poisson summation] We have
$$\sum_{n\in\Z}f(x+n)=\sum_{m\in\Z}\wh{f}(m)e^{2\pi imx}\;.$$
In particular
$$\sum_{n\in\Z}f(n)=\sum_{m\in\Z}\wh{f}(m)\;.$$
\end{proposition}

A typical application of this formula is to the ordinary Jacobi
\emph{theta function}: it is well-known (prove it otherwise) that the
function $e^{-\pi x^2}$ is invariant under Fourier transform. This
implies the following:

\begin{proposition} If $f(x)=e^{-a\pi x^2}$ for some $a>0$ then
$\wh{f}(x)=a^{-1/2}e^{-\pi x^2/a}$.
\end{proposition}

\begin{proof} Simple change of variable in the integral.\fp\end{proof}

\begin{corollary}\label{corth} Define
$$T(a)=\sum_{n\in\Z}e^{-a\pi n^2}\;.$$
We have the functional equation
$$T(1/a)=a^{1/2}T(a)\;.$$
\end{corollary}

\begin{proof} Immediate from the proposition and Poisson summation.\fp\end{proof}

This is historically the first example of modularity, which we will see in
more detail below.

\begin{exercise} Set $S=\sum_{n\ge1}e^{-(n/10)^2}$.
\begin{enumerate}\item Compute numerically $S$
to $100$ decimal digits, and show that it is apparently equal to
$5\sqrt{\pi}-1/2$.
\item Show that in fact $S$ is not exactly equal to $5\sqrt{\pi}-1/2$,
  and using the above corollary give a precise estimate for the difference.
\end{enumerate}
\end{exercise}

\begin{exercise}\begin{enumerate}
\item Show that the function $f(x)=1/\cosh(\pi x)$ is also invariant under
  Fourier transform.
\item In a manner similar to the corollary, define
  $$T_2(a)=\sum_{n\in\Z}1/\cosh(\pi na)\;.$$
  Show that we have the functional equation
  $$T_2(1/a)=aT_2(a)\;.$$
\item Show that in fact $T_2(a)=T(a)^2$ (this may be more difficult).
\item Do the same exercise as the previous one by noticing
  that $S=\sum_{n\ge1}1/\cosh(n/10)$ is very close to $5\pi-1/2$.
\end{enumerate}
\end{exercise}

Above we have mainly considered Fourier series of functions defined on $\R$.
We now consider more generally functions $f$ defined on $\C$ or a subset of
$\C$. We again assume that $f(z+1)=f(z)$, i.e., that $f$ is periodic of
period $1$. Thus (modulo regularity) $f$ has a Fourier series, but the
Fourier coefficients $a(n)$ now depend on $y=\Im(z)$:
$$f(x+iy)=\sum_{n\in\Z}a(n;y)e^{2\pi i nx}\text{\quad with\quad}a(n;y)=\int_0^1f(x+iy)e^{-2\pi inx}\,dx\;.$$
If we impose no extra condition on $f$, the \emph{functions} $a(n;y)$ are
quite arbitrary. But in almost all of our applications $f$ will be
\emph{holomorphic}; this means that
$\partial(f)(z)/\partial{\ov{z}}=0$, or equivalently that
$(\partial/\partial(x)+i\partial/\partial(y))(f)=0$. Replacing in the
Fourier expansion (recall that we do not worry about convergence issues)
gives
$$\sum_{n\in\Z}(2\pi ina(n;y)+ia'(n;y))e^{2\pi inx}=0\;,$$
hence by uniqueness of the expansion we obtain the differential equation
$a'(n;y)=-2\pi na(n;y)$, so that $a(n;y)=c(n)e^{-2\pi ny}$ for some constant
$c(n)$. This allows us to write cleanly the Fourier expansion of a
holomorphic function in the form
$$f(z)=\sum_{n\in\Z}c(n)e^{2\pi inz}\;.$$

Note that if the function is only \emph{meromorphic}, the region of convergence
will be limited by the closest pole. Consider for instance the function
$f(z)=1/(e^{2\pi iz}-1)=e^{\pi iz}/(2i\sin(\pi z))$. If we set $y=\Im(z)$
we have $|e^{2\pi iz}|=e^{-2\pi y}$, so if $y>0$ we have the Fourier expansion
$f(z)=-\sum_{n\ge0}e^{2\pi inz}$, while if $y<0$ we have the different
Fourier expansion $f(z)=\sum_{n\le-1}e^{2\pi inz}$.

\section{Elliptic Functions}

The preceding section was devoted to periodic functions. We now assume that
our functions are defined on some subset of $\C$ and assume that they
are \emph{doubly periodic}: this can be stated either by saying that there
exist two $\R$-linearly independent complex numbers $\om_1$ and $\om_2$ such
that $f(z+\om_i)=f(z)$ for all $z$ and $i=1,2$, or equivalently by saying that
there exists a \emph{lattice} $\Lambda$ in $\C$ (here $\Z\om_1+\Z\om_2$) such
that for any $\la\in\Lambda$ we have $f(z+\la)=f(z)$.

Note in passing that if $\om_1/\om_2\in\Q$ this is equivalent to (single)
periodicity, and if $\om_1/\om_2\in\R\setminus\Q$ the set of periods would be
dense so the only ``doubly periodic'' (at least continuous) functions would
essentially reduce to functions of one variable. For a similar reason there
do not exist nonconstant continuous functions which are triply periodic.

In the case of simply periodic functions considered above there already
existed some natural functions such as $e^{2\pi inx}$. In the doubly-periodic
case no such function exists (at least on an elementary level), so we have
to construct them, and for this we use the standard averaging procedure seen
and used above. Here the group is the lattice $\Lambda$, so we consider
functions of the type $f(z)=\sum_{\om\in\Lambda}\phi(z+\om)$.
For this to converge $\phi(z)$ must tend to $0$ sufficiently fast as $|z|$
tends to infinity, and since this is a double sum ($\Lambda$ is a
two-dimensional lattice), it is easy to see by comparison with an integral
(assuming $|\phi(z)|$ is regularly decreasing) that $|\phi(z)|$ should
decrease at least like $1/|z|^{\al}$ for $\al>2$.
Thus a first reasonable definition is to set
$$f(z)=\sum_{\om\in\Lambda}\dfrac{1}{(z+\om)^3}=\sum_{(m,n)\in\Z^2}\dfrac{1}{(z+m\om_1+n\om_2)^3}\;.$$
This will indeed be a doubly periodic function, and by normal convergence it
is immediate to see that it is a meromorphic function on $\C$ having only
poles for $z\in\Lambda$, so this is our first example of an
\emph{elliptic function}, which is by definition a doubly periodic function
which is meromorphic on $\C$. Note for future reference that since
$-\Lambda=\Lambda$ this specific function $f$ is odd: $f(-z)=-f(z)$.

However, this is not quite the basic elliptic function that we need. We
can integrate term by term, as long as we choose constants of integration
such that the integrated series continues to converge. To avoid stupid
multiplicative constants, we integrate $-2f(z)$: all antiderivatives of
$-2/(z+\om)^3$ are of the form $1/(z+\om)^2+C(\om)$ for some constant
$C(\om)$, hence to preserve convergence we will choose $C(0)=0$ and
$C(\om)=-1/\om^2$ for $\om\ne0$: indeed, $|1/(z+\om)^2-1/\om^2|$ is
asymptotic to $2|z|/|\om^3|$ as $|\om|\to\infty$, so we are again in
the domain of normal convergence. We will thus define:

$$\wp(z)=\dfrac{1}{z^2}+\sum_{\om\in\Lambda\setminus\{0\}}\left(\dfrac{1}{(z+\om)^2}-\dfrac{1}{\om^2}\right)\;,$$
the \emph{Weierstrass $\wp$-function}.

By construction $\wp'(z)=-2f(z)$, where $f$ is the function constructed above,
so $\wp'(z+\om)=\wp'(z)$ for any $\om\in \Lambda$, hence
$\wp(z+\om)=\wp(z)+D(\om)$ for some constant $D(\om)$ depending on $\om$ but
not on $z$. Note a slightly subtle point here: we use the fact that
$\C\setminus\Lambda$ is \emph{connected}. Do you see why?

Now as before it is clear that $\wp(z)$ is an even function: thus, setting
$z=-\om/2$ we have $\wp(\om/2)=\wp(-\om/2)+D(\om)=\wp(\om/2)+D(\om)$,
so $D(\om)=0$ hence $\wp(z+\om)=\wp(z)$ and $\wp$ is indeed an elliptic
function. There is a mistake in this reasoning: do you see it?

Since $\wp$ has poles on $\Lambda$, we cannot reason as we do when
$\om/2\in\Lambda$. Fortunately, this does not matter: since
$\om_i/2\notin\Lambda$ for $i=1,2$, we have shown at least that $D(\om_i)=0$
hence that $\wp(z+\om_i)=\wp(z)$ for $i=1,2$, so $\wp$ is doubly periodic
(so indeed $D(\om)=0$ for \emph{all} $\om\in\Lambda$).

The theory of elliptic functions is incredibly rich, and whole treatises
have been written about them. Since this course is mainly about modular
forms, we will simply summarize the main properties, and emphasize those
that are relevant to us. All are proved using manipulation of power
series and complex analysis, and all the proofs are quite straightforward.
For instance:

\begin{proposition}\label{propellf} Let $f$ be a nonzero elliptic function
  with period lattice $\Lambda$ as above, and denote by $P=P_a$ a
  ``fundamental parallelogram''
  $P_a=\{z=a+x\om_1+y\om_2,\ 0\le x<1,\ 0\le y<1\}$, where $a$ is chosen
  so that the boundary of $P_a$ does not contain any zeros or poles of $f$
  (see Figure 1).
  \begin{enumerate}\item The number of zeros of $f$ in $P$ is equal to the
    number of poles (counted with multiplicity), and this number is called
    the \emph{order} of $f$.
  \item The sum of the residues of $f$ at the poles in $P$ is equal to $0$.
  \item The sum of the zeros and poles of $f$ in $P$ belongs to $\Lambda$.
  \item If $f$ is nonconstant its order is at least $2$.
\end{enumerate}\end{proposition}

\begin{proof} For (1), (2), and (3), simply integrate $f(z)$, $f'(z)/f(z)$,
and $zf'(z)/f(z)$ along the boundary of $P$ and use the residue theorem.
For (4), we first note that by (2) $f$ cannot have order $1$ since it would
have a simple pole with residue $0$. But it also cannot have order $0$: this
would mean that $f$ has no pole, so it is an entire function, and since
it is doubly-periodic its values are those taken in $P$ which is compact,
so $f$ is \emph{bounded}. By a famous theorem of Liouville (of which this is
the no less most famous application) it implies that $f$ is constant,
contradicting the assumption of (4).\fp\end{proof}

\begin{figure}
\tikzset{->-/.style={decoration={
  markings,
  mark=at position .5 with {\arrow{latex}}},postaction={decorate}}}
\begin{tikzpicture}
    \coordinate (Origin)   at (0,0);
    \coordinate (XAxisMin) at (-3,0);
    \coordinate (XAxisMax) at (7,0);
    \coordinate (YAxisMin) at (0,-2);
    \coordinate (YAxisMax) at (0,7);
    \coordinate (A)   at (1,0);
   \draw [thin, gray,-latex] (XAxisMin) -- (XAxisMax);
    \draw [thin, gray,-latex] (YAxisMin) -- (YAxisMax);

    \clip (-3,-2) rectangle (7cm,7cm); 
    \pgftransformcm{1}{0.2}{0.5}{1}{\pgfpoint{0cm}{0cm}}
    \coordinate (Bone) at (0,3);
    \coordinate (Btwo) at (3,0);
    \draw[style=help lines,dashed] (-10,-10) grid[step=3cm] (10,10);
    \foreach \x in {-5,-4,...,5}{
      \foreach \y in {-5,-4,...,5}{
        \node[draw,circle,inner sep=1pt,gray,fill=gray] at (3*\x,3*\y) {};
      }
    }
	\node at ($0.5*(Bone)+(Btwo)+(A)$) [right] {$C_a$};
    \node[draw,circle,inner sep=1pt,black,fill] at ($(Bone)+(A)$) {};
    \node[draw,circle,inner sep=1pt,black,fill] at ($(Btwo)+(A)$) {};
    \node[draw,circle,inner sep=1pt,black,fill] at ($(Bone)+(Btwo)+(A)$) {};
    \node[draw,circle,inner sep=1pt,black,fill] at (A) {};

    \draw [ultra thick,-,black] ($(Bone)+(A)$) -- (A) node [below left] {$a$};
    \draw [ultra thick,-,black] (A)--  ($(Btwo)+(A)$) node [below right] {$\omega_2+a$};
    \draw [ultra thick,-,black]  ($(Bone)+(Btwo)+(A)$) -- ($(Bone)+(A)$) node [above] {$\omega_1+a$};
    \draw [ultra thick,-,black] ($(Btwo)+(A)$)--  ($(Btwo)+(Bone)+(A)$) node [below right] {$\omega_1+\omega_2+a$};

	 \draw (Bone) node [above left] {$\omega_1$};
	 \draw (Btwo) node [above left] {$\omega_2$};
   \end{tikzpicture}
\caption{Fundamental Parallelogram $P_a$}
\label{fig:lattice_contour}
\end{figure}
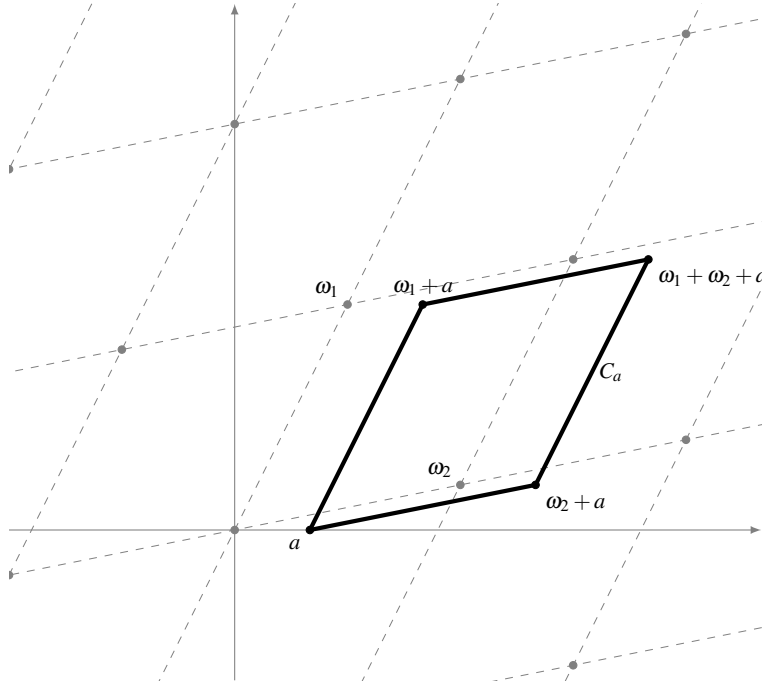

Note that clearly $\wp$ has order $2$, and the last result shows that we
cannot find an elliptic function of order $1$. Note however the following:

\begin{exercise}\begin{enumerate}
\item By integrating term by term the series defining $-\wp(z)$ show that
  if we define the \emph{Weierstrass zeta function}
  $$\z(z)=\dfrac{1}{z}+\sum_{\om\in\Lambda\setminus\{0\}}\left(\dfrac{1}{z+\om}-\dfrac{1}{\om}+\dfrac{z}{\om^2}\right)\;,$$
  this series converges normally on any compact subset of $\C\setminus\Lambda$
  and satisfies $\z'(z)=-\wp(z)$.
\item Deduce that there exist constants $\eta_1$ and $\eta_2$ such that
  $\z(z+\om_1)=\z(z)+\eta_1$ and $\z(z+\om_2)=\z(z)+\eta_2$, so that
  if $\om=m\om_1+n\om_2$ we have $\z(z+\om)=\z(z)+m\eta_1+n\eta_2$.
  Thus $\z$ (which would be of order $1$) is not doubly-periodic but only
  quasi-doubly periodic: this is called a \emph{quasi-elliptic function}.
\item By integrating around the usual fundamental parallelogram, show the
  important relation due to Legendre:
  $$\om_1\eta_2-\om_2\eta_1=\pm2\pi i\;,$$
  the sign depending on the ordering of $\om_1$ and $\om_2$.
  \end{enumerate}
\end{exercise}

The main properties of $\wp$ that we want to mention are as follows:
First, for $z$ sufficiently small and $\om\ne0$ we can expand
$$\dfrac{1}{(z+\om)^2}=\sum_{k\ge0}(-1)^k(k+1)z^k\dfrac{1}{\om^{k+2}}\;,$$
so
$$\wp(z)=\dfrac{1}{z^2}+\sum_{k\ge1}(-1)^k(k+1)z^kG_{k+2}(\Lambda)\;,$$
where we have set
$$G_k(\Lambda)=\sum_{\om\in\Lambda\setminus\{0\}}\dfrac{1}{\om^k}\;,$$
which are called \emph{Eisenstein series} of weight $k$. Since $\Lambda$
is symmetrical, it is clear that $G_k=0$ if $k$ is odd, so the expansion
of $\wp(z)$ around $z=0$ is given by
$$\wp(z)=\dfrac{1}{z^2}+\sum_{k\ge1}(2k+1)z^{2k}G_{2k+2}(\Lambda)\;.$$

Second, one can show that \emph{all} elliptic functions are simply rational
functions in $\wp(z)$ and $\wp'(z)$, so we need not look any further in our
construction.

Third, and this is probably one of the most important properties of $\wp(z)$,
it satisfies a \emph{differential equation} of order $1$: the proof is
as follows. Using the above Taylor expansion of $\wp(z)$, it is immediate
to check that
$$F(z)={\wp'(z)}^2-(4\wp(z)^3-g_2(\Lambda)\wp(z)-g_3(\Lambda))$$
has an expansion around $z=0$ beginning with $F(z)=c_1z+\cdots$,
where we have set $g_2(\Lambda)=60G_4(\Lambda)$ and
$g_3(\Lambda)=140G_6(\Lambda)$. In addition, $F$ is evidently an elliptic
function, and since it has no pole at $z=0$ it has no poles on $\Lambda$ hence
no poles at all, so it has order $0$. Thus by Proposition \ref{propellf} (4)
$f$ is constant, and since by construction it vanishes
at $0$ it is identically $0$. Thus $\wp$ satisfies the differential equation
$${\wp'(z)}^2=4\wp(z)^3-g_2(\Lambda)\wp(z)-g_3(\Lambda)\;.$$

A fourth and somewhat surprising property of the function $\wp(z)$ is
connected to the theory of \emph{elliptic curves}: the above differential
equation shows that $(\wp(z),\wp'(z))$ parametrizes the cubic curve
$y^2=4x^3-g_2x-g_3$, which is the general equation of an elliptic curve
(you do not need to know the theory of elliptic curves for what follows).
Thus, if $z_1$ and $z_2$ are in $\C\setminus\Lambda$, the two points
$P_i=(\wp(z_i),\wp'(z_i))$ for $i=1$, $2$ are on the curve, hence if we
draw the line through these two points (the tangent to the curve if they are
equal), it is immediate to see from Proposition \ref{propellf} (3) that the
third point of intersection corresponds to the
parameter $-(z_1+z_2)$, and can of course be computed as a rational
function of the coordinates of $P_1$ and $P_2$. It follows that
$\wp(z)$ (and $\wp'(z)$) possesses an \emph{addition formula}
expressing $\wp(z_1+z_2)$ in terms of the $\wp(z_i)$ and $\wp'(z_i)$.

\begin{exercise} Find this addition formula. You will have to distinguish
the cases $z_1=z_2$, $z_1=-z_2$, and $z_1\ne\pm z_2$.\end{exercise}

An interesting corollary of the differential equation for $\wp(z)$, which we
will prove in a different way below, is a \emph{recursion} for the
Eisenstein series $G_{2k}(\Lambda)$:

\begin{proposition} We have the recursion for $k\ge4$:
$$(k-3)(2k-1)(2k+1)G_{2k}=3\sum_{2\le j\le k-2}(2j-1)(2(k-j)-1)G_{2j}G_{2(k-j)}\;.$$
\end{proposition}

\begin{proof} Taking the derivative of the differential equation and dividing by
$2\wp'$ we obtain $\wp''(z)=6\wp(z)^2-g_2(\Lambda)/2$. If we set by
convention $G_0(\Lambda)=-1$ and $G_2(\Lambda)=0$, and for notational
simplicity omit $\Lambda$ which is fixed, we have
$\wp(z)=\sum_{k\ge-1}(2k+1)z^{2k}G_{2k+2}$, so on the one hand
$$\wp''(z)=\sum_{k\ge-1}(2k+1)(2k)(2k-1)z^{2k-2}G_{2k+2}\;,$$
and on the other hand $\wp(z)^2=\sum_{K\ge-2}a(K)z^{2K}$ with
$$a(K)=\sum_{k_1+k_2=K}(2k_1+1)(2k_2+1)G_{2k_1+2}G_{2k_2+2}\;.$$
Replacing in the differential equation it is immediate to check that the
coefficients agree up to $z^2$, and for $K\ge2$ we have the identification
$$6\sum_{\substack{k_1+k_2=K\\k_i\ge-1}}(2k_1+1)(2k_2+1)G_{2k_1+2}G_{2k_2+2}=(2K+3)(2K+2)(2K+1)G_{2K+4}$$
which is easily seen to be equivalent to the recursion of the proposition
using $G_0=-1$ and $G_2=0$.\fp\end{proof}

For instance
$$G_8=\dfrac{3}{7}G_4^2\,\quad G_{10}=\dfrac{5}{11}G_4G_6\,\quad G_{12}=\dfrac{18G_4^3+25G_6^2}{143}\;,$$
and more generally this implies that $G_{2k}$ is a \emph{polynomial} in
$G_4$ and $G_6$ with rational coefficients which are \emph{independent} of
the lattice $\Lambda$.

As other corollary, we note that if we choose $\om_2=1$ and $\om_1=iT$ with
$T$ tending to $+\infty$, then the definition
$G_{2k}(\Lambda)=\sum_{(m,n)\in\Z^2\setminus\{(0,0)\}}(m\om_1+n\om_2)^{-2k}$ implies
that $G_{2k}(\Lambda)$ will tend to $\sum_{n\in\Z\setminus\{0\}}n^{-2k}=2\z(2k)$,
where $\z$ is the Riemann zeta function. If follows that for all $k\ge2$,
$\z(2k)$ is a polynomial in $\z(4)$ and $\z(6)$ with rational coefficients.
Of course this is a weak but nontrivial result, since we know that $\z(2k)$
is a rational multiple of $\pi^{2k}$.

\smallskip

To finish this section on elliptic functions and make the transition to modular
forms, we write explicitly $\Lambda=\Lambda(\om_1,\om_2)$ and by abuse of
notation $G_{2k}(\om_1,\om_2):=G_{2k}(\Lambda(\om_1,\om_2))$, and we consider
the dependence of $G_{2k}$ on $\om_1$ and $\om_2$. We note two
evident facts: first, $G_{2k}(\om_1,\om_2)$ is \emph{homogeneous} of
degree $-2k$: for any nonzero complex number $\la$ we have
$G_{2k}(\la\om_1,\la\om_2)=\la^{-2k}G_{2k}(\om_1,\om_2)$. In particular,
$G_{2k}(\om_1,\om_2)=\om_2^{-2k}G_{2k}(\om_1/\om_2,1)$.
Second, a general $\Z$-basis of $\Lambda$ is given by
$(\om'_1,\om'_2)=(a\om_1+b\om_2,c\om_1+d\om_2)$ with $a$, $b$, $c$, $d$
integers such that $ad-bc=\pm1$. If we choose an \emph{oriented} basis
such that $\Im(\om_1/\om_2)>0$ we in fact have $ad-bc=1$.

Thus, $G_{2k}(a\om_1+b\om_2,c\om_1+d\om_2)=G_{2k}(\om_1,\om_2)$, and using
homogeneity this can be written
$$(c\om_1+d\om_2)^{-2k}G_{2k}\left(\dfrac{a\om_1+b\om_2}{c\om_1+d\om_2},1\right)=\om_2^{-2k}G_{2k}\left(\dfrac{\om_1}{\om_2},1\right)\;.$$
Thus, if we set $\tau=\om_1/\om_2$ and by an additional abuse of notation
abbreviate $G_{2k}(\tau,1)$ to $G_{2k}(\tau)$, we have by definition
$$G_{2k}(\tau)=\sum_{(m,n)\in\Z^2\setminus\{(0,0)\}}(m\tau+n)^{-2k}\;,$$
and we have shown the following \emph{modularity} property:

\begin{proposition} For any $\psmm{a}{b}{c}{d}\in\SL_2(\Z)$, the group of
  $2\times2$ integer matrices of determinant $1$, and any $\tau\in\C$ with
  $\Im(\tau)>0$ we have
  $$G_{2k}\left(\dfrac{a\tau+b}{c\tau+d}\right)=(c\tau+d)^{2k}G_{2k}(\tau)\;.$$
\end{proposition}

This will be our basic definition of (weak) modularity.

\section{Modular Forms and Functions}

\subsection{Definitions}

Let us introduce some notation:

$\bullet$ We denote by $\G$ the \emph{modular group $\SL_2(\Z)$}. Note that
properly speaking the modular group should be the group of transformations
$\tau\mapsto(a\tau+b)/(c\tau+d)$, which is isomorphic to the quotient of
$\SL_2(\Z)$ by the equivalence relation saying that $M$ and $-M$ are
equivalent, but for this course we will stick to this definition. If
$\ga=\psmm{a}{b}{c}{d}$ we will of course write $\ga(\tau)$ for
$(a\tau+b)/(c\tau+d)$.

$\bullet$ The \emph{Poincar\'e upper half-plane} $\H$ is the set of complex
numbers $\tau$ such that $\Im(\tau)>0$. Since for $\ga=\psmm{a}{b}{c}{d}\in\G$
we have $\Im(\ga(\tau))=\Im(\tau)/|c\tau+d|^2$, we see that $\G$ is a group of
transformations of $\H$ (more generally so is $\SL_2(\R)$, there is nothing
special about $\Z$).

$\bullet$ The completed upper half-plane $\ov{\H}$ is by definition
$\ov{\H}=\H\cup\P_1(\Q)=\H\cup\Q\cup\{i\infty\}$. Note that this is
\emph{not} the closure in the topological sense, since we do not
include any real irrational numbers.

\begin{definition} Let $k\in\Z$ and let $F$ be a function from $\H$ to $\C$.
  \begin{enumerate}\item We will say that
    $F$ is \emph{weakly modular} of weight $k$ for $\G$ if for all
    $\ga=\psmm{a}{b}{c}{d}\in\G$ and all $\tau\in\H$ we have
    $$F(\ga(\tau))=(c\tau+d)^kF(\tau)\;.$$
  \item We will say that $F$ is a modular \emph{form} if, in addition,
    $F$ is holomorphic on $\H$ and if $|F(\tau)|$ remains bounded as
    $\Im(\tau)\to\infty$.
  \item We will say that $F$ is a modular \emph{cusp form} if it is
    a modular form such that $F(\tau)$ tends to $0$ as $\Im(\tau)\to\infty$.
  \end{enumerate}
\end{definition}

We make a number of immediate but important remarks.

\begin{remarks}{\rm \begin{enumerate}
\item The Eisenstein series $G_{2k}(\tau)$ are basic examples of modular
  forms of weight $2k$, which are not cusp forms since $G_{2k}(\tau)$
  tends to $2\z(2k)\ne0$ when $\Im(\tau)\to\infty$.
\item With the present definition, it is clear that there are no nonzero
  modular forms of \emph{odd weight} $k$, since if $k$ is odd we have
  $(-c\tau-d)^k=-(c\tau+d)^k$ and $\ga(\tau)=(-\ga)(\tau)$. However, when
  considering modular forms defined on \emph{subgroups} of $\G$ there may
  be modular forms of odd weight, so we keep the above definition.
\item Applying modularity to $\ga=T=\psmm{1}{1}{0}{1}$ we see that
  $F(\tau+1)=F(\tau)$, hence $F$ has a Fourier series expansion, and if
  $F$ is holomorphic, by the remark made above in the section on Fourier
  series, we have an expansion $F(\tau)=\sum_{n\in\Z}a(n)e^{2\pi in\tau}$
  with $a(n)=e^{2\pi ny}\int_0^1F(x+iy)e^{-2\pi inx}\,dx$ for any $y>0$.
  Thus, if $|F(x+iy)|$ remains bounded as $y\to\infty$ it follows that
  as $y\to\infty$ we have $a(n)\le Be^{2\pi ny}$ for a suitable constant $B$,
  so we deduce that $a(n)=0$ whenever $n<0$ since $e^{2\pi ny}\to0$. Thus
  if $F$ is a modular \emph{form} we have
  $F(\tau)=\sum_{n\ge0}a(n)e^{2\pi in\tau}$, hence
  $\lim_{\Im(\tau)\to\infty}F(\tau)=a(0)$, so $F$ is a cusp form if
  and only if $a(0)=0$.
\end{enumerate}}
\end{remarks}

\begin{definition} We will denote by $M_k(\G)$ the vector space of modular
forms of weight $k$ on $\G$ ($M$ for Modular of course), and by $S_k(\G)$
the subspace of cusp forms ($S$ for the German Spitzenform, meaning exactly
cusp form).\end{definition}

Notation: for any matrix $\ga=\psmm{a}{b}{c}{d}$ with $ad-bc>0$, we will
define the weight $k$ \emph{slash operator} $F|_k\ga$ by
$$F|_k\ga(\tau)=(ad-bc)^{k/2}(c\tau+d)^{-k}F(\ga(\tau))\;.$$
The reason for the factor $(ad-bc)^{k/2}$ is that $\la\ga$ has the same
action on $\H$ as $\ga$, so this makes the formula homogeneous.
For instance, $F$ is weakly modular of weight $k$ if and only if
$F|_k\ga=F$ for all $\ga\in\G$.

We will also use the universal modular form convention of
writing $q$ for $e^{2\pi i\tau}$, so that a Fourier expansion is of the type
$F(\tau)=\sum_{n\ge0}a(n)q^n$. We use the additional convention that if
$\al$ is any complex number, $q^{\al}$ will mean $e^{2\pi i\tau\al}$.

\begin{exercise} Let $F(\tau)=\sum_{n\ge0}a(n)q^n\in M_k(\G)$, and let
$\ga=\psmm{A}{B}{C}{D}$ be a matrix in $M_2^+(\Z)$, i.e., $A$, $B$, $C$, and
$D$ are integers and $\Delta=\det(\ga)=AD-BC>0$. Set $g=\gcd(A,C)$, let $u$ and
$v$ be such that $uA+vC=g$, set $b=uB+vD$, and finally let
$\z_{\Delta}=e^{2\pi i/\Delta}$. Prove the matrix identity
$$\begin{pmatrix}A&B\\C&D\end{pmatrix}=\begin{pmatrix}A/g&-v\\C/g&u\end{pmatrix}\begin{pmatrix}g&b\\0&\Delta/g\end{pmatrix}\;,$$
and deduce that we have the more general Fourier expansion
$$F|_k\ga(\tau)=\dfrac{g^{k/2}}{\Delta^k}\sum_{n\ge0}\z_{\Delta}^{nbg}a(n)q^{g^2/\Delta}\;,$$
which is of course equal to $F$ if $\Delta=1$, since then $g=1$.
\end{exercise}

\subsection{Basic Results}

The first fundamental result in the theory of modular forms is that these
spaces are \emph{finite-dimensional}. The proof uses exactly the same
method that we have used to prove the basic results on elliptic functions.
We first note that there is a ``fundamental domain'' (which replaces the
fundamental parallelogram) for the action of $\G$ on $\H$, given by
$$\FF=\{\tau\in\H,\ -1/2\le\Re(\tau)<1/2,\ |\tau|\ge1\}\;.$$

\begin{figure}
\centering 
\begin{tikzpicture}[scale=2]
	\def\b{0.866025403784439};
    \tikzstyle{every node}=[font=\footnotesize]
	\coordinate (Origin)   at (0,0);
    \coordinate (XAxisMin) at (-1.5,0);
    \coordinate (XAxisMax) at (1.5,0);
    \coordinate (YAxisMin) at (0,0);
    \coordinate (YAxisMax) at (0,3);

    \coordinate (e1) at (0.5,\b);
    \coordinate (e2) at (-0.5,\b);
    \coordinate (f1) at (0.5,\b+3.5);
    \coordinate (f2) at (-0.5,\b+3.5);
    \clip (-1,-0.3) rectangle (2,2.5); 

    \draw (0.5,-0.05) -- (0.5,0.05);
    \draw (0.5,0) node [below] {$\frac{1}{2}$};
    \draw (-0.5,-0.05) -- (-0.5,0.05);
    \draw (-0.55,0) node [below] {$-\frac{1}{2}$};
    \draw [thin, gray,-latex] (XAxisMin) -- (XAxisMax);
    \draw [thin, gray,-latex] (YAxisMin) -- (YAxisMax);
	\draw (0,1.5) node [above right,font=\normalsize] {$\FF$};
    \shade [gray,bottom color=transparent!100, top color=transparent!5,fill opacity=0.02] 
       (f1) -- (e1) arc (60:120:1) -- (f2) -- cycle;
    \draw [black] 
       (f1) -- (e1) arc (60:120:1) -- (f2);
   \end{tikzpicture}
\caption{The fundamental domain, $\FF$, of $\G$}
\label{fig:fundamental_domain_gamma}
\end{figure}
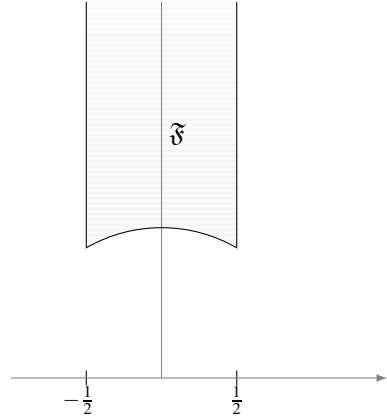

The proof that this is a fundamental domain, in other words that any
$\tau\in\H$ has a unique image by $\G$ belonging to $\FF$ is not very difficult
and will be omitted. We then integrate $F'(z)/F(z)$ along the boundary of
$\FF$, and using modularity we obtain the following result:

\begin{theorem}\label{thmval} Let $F\in M_k(\G)$ be a nonzero modular form.
  For any
  $\tau_0\in\H$, denote by $v_{\tau_0}(F)$ the \emph{valuation} of $F$ at
  $\tau_0$, i.e., the unique integer $v$ such that $F(\tau)/(\tau-\tau_0)^v$
  is holomorphic and nonzero at $\tau_0$, and if $F(\tau)=G(e^{2\pi i\tau})$,
  define $v_{i\infty}(F)=v_0(G)$ (i.e., the number of first vanishing Fourier
  coefficients of $F$). We have the formula
  $$v_{i\infty}(F)+\sum_{\tau\in\FF}\dfrac{v_{\tau}(F)}{e_{\tau}}=\dfrac{k}{12}\;,$$
  where $e_i=2$, $e_{\rho}=3$, and $e_{\tau}=1$ otherwise ($\rho=e^{2\pi i/3}$).
\end{theorem}

This theorem has many important consequences but, as already noted, the most
important is that it implies that $M_k(\G)$ is finite dimensional. First,
it trivially implies that $k\ge0$, i.e., there are no modular \emph{forms}
of negative weight. In addition it easily implies the following:

\begin{corollary}\label{cordim} Let $k\ge0$ be an even integer. We have
  \begin{align*}\dim(M_k(\G))&=\begin{cases}
  \lfloor k/12\rfloor&\text{\quad if $k\equiv2\pmod{12}$\;,}\\
  \lfloor k/12\rfloor+1&\text{\quad if $k\not\equiv2\pmod{12}$\;,}
  \end{cases}\\
  \dim(S_k(\G))&=\begin{cases}
  0&\text{\quad if $k<12$\;,}\\
  \lfloor k/12\rfloor-1&\text{\quad if $k\ge12$, $k\equiv2\pmod{12}$\;,}\\
  \lfloor k/12\rfloor&\text{\quad if $k\ge12$, $k\not\equiv2\pmod{12}$\;.}
  \end{cases}\end{align*}
\end{corollary}

Since the product of two modular forms is clearly a modular form (of weight
the sum of the two weights), It is clear that $M_{*}(\G)=\bigoplus_kM_k(\G)$
(and similarly $S_{*}(\G)$) is an algebra, whose structure is easily described:

\begin{corollary}\label{core4e6} We have $M_{*}(\G)=\C[G_4,G_6]$, and
  $S_{*}(\G)=\Delta M_{*}(\G)$, where $\Delta$ is the unique
  generator of the one-dimensional vector space $S_{12}(\G)$ whose Fourier
  expansion begins with $\Delta=q+O(q^2)$.
\end{corollary}

Thus, for instance, $M_0(\G)=\C$, $M_2(\G)=\{0\}$, $M_4(\G)=\C G_4$,
$M_6(\G)=\C G_6$, $M_8(\G)=\C G_8=\C G_4^2$, $M_{10}(\G)=\C G_{10}=\C G_4G_6$,
$$M_{12}(\G)=\C G_{12}\oplus\C\Delta=\C G_4^3\oplus\C G_6^2\;.$$
In particular, we recover the fact proved differently that $G_8$ is a multiple
of $G_4^2$ (the exact multiple being obtained by computing the Fourier
expansions), $G_{10}$ is a multiple of $G_4G_6$, $G_{12}$ is a linear
combination of $G_4^3$ and $G_6^2$. Also, we see that $\Delta$ is
a linear combination of $G_4^3$ and $G_6^2$ (we will see this more precisely
below).

A basic result on the structure of the modular group $\G$ is the following:

\begin{proposition} Set $T=\psmm{1}{1}{0}{1}$, which acts on $\H$ by the
  unit translation $\tau\mapsto\tau+1$, and $S=\psmm{0}{-1}{1}{0}$ which
  acts on $\H$ by the symmetry-inversion $\tau\mapsto-1/\tau$.
  Then $\G$ is generated by $S$ and $T$, with relations generated by
  $S^2=-I$ and $(ST)^3=-I$ ($I$ the identity matrix).
\end{proposition}

There are several (easy) proofs of this fundamental result, which we do not
give. Simply note that this proposition is essentially equivalent to the
fact that the set $\FF$ described above is indeed a fundamental domain.

A consequence of this proposition is that to check whether some function
$F$ has the modularity property, it is sufficient to check that
$F(\tau+1)=F(\tau)$ and $F(-1/\tau)=\tau^kF(\tau)$.

\begin{exercise}\label{ex:Bol} (Bol's identity). Let $F$ be any continuous
function defined on the upper-half plance $\H$, and define $I_0(F,a)=F$ and
for any integer $m\ge1$ and $a\in\ov{\H}$ set:
$$I_m(F,a)(\tau)=\int_a^{\tau}\dfrac{(\tau-z)^{m-1}}{(m-1)!}F(z)\,dz\;.$$
\begin{enumerate}\item Show that $I_m(F,a)'(\tau)=I_{m-1}(F,a)(\tau)$, so that
  $I_m(F,a)$ is an $m$th antiderivative of $F$.
\item Let $\ga\in\G$, and assume that $k\ge1$ is an integer. Show that
  $$I_{k-1}(F,a)|_{2-k}\ga=I_{k-1}(F|_k\ga,\ga^{-1}(a))\;.$$
\item Deduce that if we set $F^*_a=I_{k-1}(F,a)$ then
  $$D^{(k-1)}(F^*_a|_{2-k}\ga)=F|_k\ga\;,$$
where $D=(1/2\pi i)d/d\tau=qd/dq$ is the basic differential operator that
we will use (see Section \ref{sec:deriv}).
\item Assume now that $F$ is weakly modular of weight $k\ge1$ and holomorphic
  on $\H$ (in particular if $F\in M_k(\G)$, but $|F|$ could be unbounded as
  $\Im(\tau)\to\infty$). Show that
  $$(F^*_a|_{2-k}|\ga)(\tau)=F^*_a(\tau)+P_{k-2}(\tau)\;,$$
    where $P_{k-2}$ is the polynomial of degree less than or equal to $k-2$
    given by
    $$P_{k-2}(X)=\int_{\ga^{-1}(a)}^a\dfrac{(X-z)^{k-2}}{(k-2)!}F(z)\,dz\;.$$
\end{enumerate}
\end{exercise}

What this exercise shows is that the $(k-1)$st derivative of some function
which behaves modularly in weight $2-k$ behaves modularly in weight $k$,
and conversely that the $(k-1)$st antiderivative of some function which
behaves modularly in weight $k$ behaves modularly in weight $k$ up to
addition of a polynomial of degree at most $k-2$. This duality between
weights $k$ and $2-k$ is in fact a consequence of the \emph{Riemann--Roch
theorem}.

Note also that this exercise is the beginning of the fundamental theories
of \emph{periods} and of \emph{modular symbols}.

Also, it is not difficult to generalize Bol's identity. For instance, applied
to the Eisenstein series $G_4$ and using Proposition \ref{propg4} below
we obtain:

\begin{proposition}\label{propg4star}\begin{enumerate}
  \item Set $$F_4^*(\tau)=-\dfrac{\pi^3}{180}\left(\dfrac{\tau}{i}\right)^3+\sum_{n\ge1}\sigma_{-3}(n)q^n\;.$$
    We have the functional equation
    $$\tau^2F_4^*(-1/\tau)=F_4^*(\tau)+\dfrac{\zeta(3)}{2}(1-\tau^2)-\dfrac{\pi^3}{36}\dfrac{\tau}{i}\;.$$
  \item Equivalently, if we set
    $$F_4^{**}(\tau)=-\dfrac{\pi^3}{180}\left(\dfrac{\tau}{i}\right)^3-\dfrac{\pi^3}{72}\left(\dfrac{\tau}{i}\right)+\dfrac{\z(3)}{2}+\sum_{n\ge1}\sigma_{-3}(n)q^n$$
    we have the functional equation
    $$F_4^{**}(-1/\tau)=\tau^{-2}F_4^{**}(\tau)\;.$$
  \end{enumerate}
\end{proposition}

Note that the appearance of $\zeta(3)$ comes from the fact that, up to
a multiplicative constant, the $L$-function associated to $G_4$ is equal
to $\z(s)\z(s-3)$, whose value at $s=3$ is equal to $-\zeta(3)/2$.

\subsection{The Scalar Product}

We begin by the following exercise:

\begin{exercise}\label{ex:dmu}\begin{enumerate}
  \item Denote by $d\mu=dxdy/y^2$ a measure on $\H$, where as usual
  $x$ and $y$ are the real and imaginary part of $\tau\in\H$. Show that this
    measure is invariant under $\SL_2(\R)$.
  \item Let $f$ and $g$ be in $M_k(\G)$. Show that the function
    $F(\tau)=f(\tau)\ov{g(\tau)}y^k$ is invariant under the modular group $\G$.
  \end{enumerate}
\end{exercise}

It follows in particular from this exercise that if $F(\tau)$ is any
integrable function which is invariant by the modular group $\G$, the
integral $\int_{\G\backslash\H}F(\tau)d\mu$ makes sense if it converges.
Since $\FF$ is a fundamental domain for the action of $\G$ on $\H$, this
can also be written $\int_{\FF}F(\tau)d\mu$. Thus it follows from the second
part that we can define
$$<f,g>=\int_{\G\backslash\H}f(\tau)\ov{g(\tau)}y^k\,\dfrac{dxdy}{y^2}\;,$$
whenever this converges.

It is immediate to show that a necessary and sufficient condition for
convergence is that at least one of $f$ and $g$ be a cusp form, i.e., lies in
$S_k(\G)$. In particular it is clear that this defines a \emph{scalar
product} on $S_k(\G)$ called the Petersson scalar product.
In addition, any cusp form in $S_k(\G)$ is \emph{orthogonal} to $G_k$
with respect to this scalar product. It is instructive to give a sketch
of the simple proof of this fact:

\begin{proposition}\label{prop:unfold} If $f\in S_k(\G)$ we have
  $<G_k,f>=0$.\end{proposition}

\begin{proof} Recall that
  $G_k(\tau)=\sum_{(m,n)\in\Z^2\setminus\{(0,0)\}}(m\tau+n)^{-k}$. We
  split the sum according to the GCD of $m$ and $n$: we let $d=\gcd(m,n)$,
  so that $m=dm_1$ and $n=dn_1$ with $\gcd(m_1,n_1)=1$. It follows that
  $$G_k(\tau)=2\sum_{d\ge1}d^{-k}E_k(\tau)=2\z(k)E_k(\tau)\;,$$
  where $E_k(\tau)=(1/2)\sum_{\gcd(m,n)=1}(m\tau+n)^{-k}$. We thus need to
  prove that $<E_k,f>=0$.

  On the other hand, denote by $\G_\infty$ the group generated by $T$, i.e.,
  translations $\psmm{1}{b}{0}{1}$ for $b\in\Z$. This acts by left
  multiplication on $\G$, and it is immediate to check that a system of
  representatives for this action is given by matrices $\psmm{u}{v}{m}{n}$,
  where $\gcd(m,n)=1$ and $u$ and $v$ are chosen arbitrarily (but only once
  for each pair $(m,n)$) such that $un-vm=1$. It follows that we can
  write
  $$E_k(\tau)=\sum_{\ga\in\G_\infty\backslash\G}(m\tau+n)^{-k}\;,$$
  where it is understood that $\ga=\psmm{u}{v}{m}{n}$
  (the factor $1/2$ has disappeared since $\ga$ and $-\ga$ have the
  same action on $\H$).

  Thus
  \begin{align*}<E_k,f>&=\int_{\G\backslash\H}\sum_{\ga\in\G_\infty\backslash\G}(m\tau+n)^{-k}\ov{f(\tau)}y^k\,\dfrac{dxdy}{y^2}\\
    &=\sum_{\ga\in\G_\infty\backslash\G}\int_{\G\backslash\H}(m\tau+n)^{-k}\ov{f(\tau)}y^k\,\dfrac{dxdy}{y^2}\;.\end{align*}

  Now note that by modularity $f(\tau)=(m\tau+n)^{-k}f(\ga(\tau))$, and
  since $\Im(\ga(\tau))=\Im(\tau)/|m\tau+n|^2$ it follows that
  $$(m\tau+n)^{-k}\ov{f(\tau)}y^k=\ov{f(\ga(\tau))}\Im(\ga(\tau))^k\;.$$
  Thus, since $d\mu=dxdy/y^2$ is an invariant measure we have
  \begin{align*}<E_k,f>&=\sum_{\ga\in\G_\infty\backslash\G}\int_{\G\backslash\H}\ov{f(\ga(\tau))}\Im(\ga(\tau))^kd\mu
    =\int_{\G_\infty\backslash\H}\ov{f(\tau)}y^k\,\dfrac{dxdy}{y^2}\;.\end{align*}

  Since $\G_\infty$ is simply the group of integer translations, a fundamental
  domain for $\G_\infty\backslash\H$ is simply the vertical strip
  $[0,1]\times[0,\infty[$, so that
  $$<E_k,f>=\int_0^\infty y^{k-2}dy\int_0^1\ov{f(x+iy)}dx\;,$$
  which trivially vanishes since the inner integral is simply the
  conjugate of the constant term in the Fourier expansion of $f$,
  which is $0$ since $f\in S_k(\G)$.\end{proof}

The above procedure (replacing the complicated fundamental domain of
$\G\backslash\H$ by the trivial one of $\G_\infty\backslash\H$) is very common
in the theory of modular forms and is called \emph{unfolding}.

\subsection{Fourier Expansions}

The Fourier expansions of the Eisenstein series $G_{2k}(\tau)$ are easy to
compute. The result is the following:

\begin{proposition}\label{propg4} For $k\ge4$ even we have the Fourier
expansion
  $$G_k(\tau)=2\zeta(k)+2\dfrac{(2\pi i)^{k}}{(k-1)!}\sum_{n\ge1}\sigma_{k-1}(n)q^n\;,$$
  where $\sigma_{k-1}(n)=\sum_{d\mid n,\ d>0}d^{k-1}$.
\end{proposition}

Since we know that when $k$ is even
$2\zeta(k)=-(2\pi i)^kB_k/k!$, where $B_k$ is the $k$-th Bernoulli number
defined by
$$\dfrac{t}{e^t-1}=\sum_{k\ge0}\dfrac{B_k}{k!}t^k\;,$$
it follows that $G_k=2\zeta(k)E_k$, with
$$E_k(\tau)=1-\dfrac{2k}{B_k}\sum_{n\ge1}\sigma_{k-1}(n)q^n\;.$$
This is the normalization of Eisenstein series that we will use. For instance
\begin{align*}
  E_4(\tau)&=1+240\sum_{n\ge1}\sigma_3(n)q^n\;,\\
  E_6(\tau)&=1-504\sum_{n\ge1}\sigma_5(n)q^n\;,\\
  E_8(\tau)&=1+480\sum_{n\ge1}\sigma_7(n)q^n\;.\end{align*}
In particular, the relations given above which follow from the dimension
formula become much simpler and are obtained simply by looking at the
first terms in the Fourier expansion:
$$E_8=E_4^2\;,\quad E_{10}=E_4E_6\;,\quad E_{12}=\dfrac{441E_4^3+250E_6^2}{691}\;,\quad\Delta=\dfrac{E_4^3-E_6^2}{1728}\;.$$

Note that the relation $E_4^2=E_8$ (and the others) implies a highly nontrivial
relation between the sum of divisors function: if we set by convention
$\sigma_3(0)=1/240$, so that $E_4(\tau)=\sum_{n\ge0}\sigma_3(n)q^n$, we have
$$E_8(\tau)=E_4^2(\tau)=240^2\sum_{n\ge0}q^n\sum_{0\le m\le n}\sigma_3(m)\sigma_3(n-m)\;,$$
so that by identification
$\sigma_7(n)=120\sum_{0\le m\le n}\sigma_3(m)\sigma_3(n-m)$, so
$$\sigma_7(n)=\sigma_3(n)+120\sum_{1\le m\le n-1}\sigma_3(m)\sigma_3(n-m)\;.$$
It is quite difficult (but not impossible) to prove this directly, i.e.,
without using at least indirectly the theory of modular forms.

\begin{exercise} Find a similar relation for $\sigma_9(n)$ using
$E_{10}=E_4E_6$.\end{exercise}

This type of reasoning is one of the reasons for which the theory of modular
forms is so important (and lots of fun!): if you have a modular form $F$, you
can usually express it in terms of a completely explicit basis of the space to
which it belongs since spaces of modular forms are \emph{finite-dimensional}
(in the present example, the space is one-dimensional), and deduce
highly nontrivial relations for the Fourier coefficients. We will see a further
example of this below for the number $r_k(n)$ of representations of an integer
$n$ as a sum of $k$ squares.

\begin{exercise}\begin{enumerate}\item Prove that for any $k\in\C$ we have the
identity
$$\sum_{n\ge1}\sigma_k(n)q^n=\sum_{n\ge1}\dfrac{n^kq^n}{1-q^n}\;,$$
the right-hand side being called a \emph{Lambert series}.
\item Set $F(k)=\sum_{n\ge1}n^k/(e^{2\pi n}-1)$. Using the Fourier expansions
given above, compute explicitly $F(5)$ and $F(9)$.
\item Using Proposition \ref{propg4star}, compute explicitly $F(-3)$.
\item Using Proposition \ref{prop:E2} below, compute explicitly $F(1)$.
\end{enumerate}\end{exercise}

Note that in this exercise we only compute $F(k)$ for $k\equiv1\pmod4$. It is
also possible but more difficult to compute $F(k)$ for $k\equiv3\pmod4$. For
instance we have:
$$F(3)=\dfrac{\Gamma(1/4)^8}{80(2\pi)^6}-\dfrac{1}{240}\;.$$

\subsection{Obtaining Modular Forms by Averaging}

We have mentioned at the beginning of this course that one of the ways to
obtain functions satisfying functional equations is to use \emph{averaging}
over a suitable group or set: we have seen this for periodic functions in
the form of the Poisson summation formula, and for doubly-periodic
functions in the construction of the Weierstrass $\wp$-function.
We can do the same for modular forms, but we must be careful in two different
ways. First, we do not want \emph{invariance} by $\G$, but we want an
automorphy factor $(c\tau+d)^k$. This is easily dealt with by noting that
$(d/d\tau)(\ga(\tau))=(c\tau+d)^{-2}$: indeed, if $\phi$ is some function on
$\H$ we can define
$$F(\tau)=\sum_{\ga\in\G}\phi(\ga(\tau))((d/d\tau)(\ga(\tau)))^{k/2}\;.$$

\begin{exercise} Ignoring all convergence questions, by using the chain rule
$(f\circ g)'=(f'\circ g)g'$ show that for all $\delta=\psmm{A}{B}{C}{D}\in\G$
we have
$$F(\delta(\tau))=(C\tau+D)^kF(\tau)\;.$$
\end{exercise}

But the second important way in which we must be careful is that the
above contruction rarely converges. There are, however, examples where it
does converge:

\begin{exercise} Let $\phi(\tau)=\tau^{-m}$, so that
$$F(\tau)=\sum_{\ga=\psmm{a}{b}{c}{d}\in\G}\dfrac{1}{(a\tau+b)^m(c\tau+d)^{k-m}}\;.$$
Show that if $2\le m\le k-2$ and $m\ne k/2$ this series converges
normally on any compact subset of $\H$ (i.e., it is majorized by a convergent
series with positive terms), so defines a modular form in $M_k(\G)$.
\end{exercise}

Note that the series converges also for $m=k/2$, but this is more difficult.

One of the essential reasons for non-convergence of the function $F$ is
the trivial observation that for a given pair of coprime integers $(c,d)$
there are infinitely many elements $\ga\in\G$ having $(c,d)$ as their
second row. Thus in general it seems more reasonable to define
$$F(\tau)=\sum_{\gcd(c,d)=1}\phi(\ga_{c,d}(\tau))(c\tau+d)^{-k}\;,$$
where $\ga_{c,d}$ is \emph{any fixed} matrix in $\G$ with second row equal to
$(c,d)$. However, we need this to make sense: if
$\ga_{c,d}=\psmm{a}{b}{c}{d}\in\G$ is one such matrix, it is clear that the
general matrix having second row equal to $(c,d)$ is
$T^n\psmm{a}{b}{c}{d}=\psmm{a+nc}{b+nd}{c}{d}$, and as usual
$T=\psmm{1}{1}{0}{1}$ is translation by $1$: $\tau\mapsto\tau+1$. Thus, an
essential necessary condition for our series to make any kind of sense is
that the function $\phi$ be \emph{periodic} of period~$1$.

The simplest such function is of course the constant function $1$:

\begin{exercise} (See the proof of Proposition \ref{prop:unfold}.) Show that 
$$F(\tau)=\sum_{\gcd(c,d)=1}(c\tau+d)^{-k}=2E_k(\tau)\;,$$
where $E_k$ is the normalized Eisenstein series defined above.
\end{exercise}

But by the theory of Fourier series, we know that periodic functions of
period $1$ are (infinite) linear combinations of the functions
$e^{2\pi i n\tau}$. This leads to the definition of \emph{Poincar\'e series}:

$$P_k(n;\tau)=\dfrac{1}{2}\sum_{\gcd(c,d)=1}\dfrac{e^{2\pi in\ga_{c,d}(\tau)}}{(c\tau+d)^k}\;,$$
where we note that we can choose any matrix $\ga_{c,d}$ with bottom row
$(c,d)$ since the function
$e^{2\pi in\tau}$ is $1$-periodic, so that $P_k(n;\tau)\in M_k(\G)$.

\begin{exercise} Assume that $k\ge4$ is even.
\begin{enumerate}
\item Show that if $n<0$ the series defining $P_k$ diverges (wildly in fact).
\item Note that $P_k(0;\tau)=E_k(\tau)$, so that
$\lim_{\tau\to i\infty}P_k(0;\tau)=1$.
  Show that if $n>0$ the series converges normally and
  that we have $\lim_{\tau\to i\infty}P_k(n;\tau)=0$. Thus in fact
  $P_k(n;\tau)\in S_k(\G)$ if $n>0$.
\item By using the same \emph{unfolding method} as in Proposition
  \ref{prop:unfold}, show that if $f=\sum_{n\ge0}a(n)q^n\in M_k(\G)$
  and $n>0$ we have
  $$<P_k(n),f>=\dfrac{(k-2)!}{(4\pi n)^{k-1}}a(n)\;.$$
\end{enumerate}
\end{exercise}
  
It is easy to show that in fact the $P_k(n)$ \emph{generate} $S_k(\G)$.
We can also compute their \emph{Fourier expansions} as we have done for
$E_k$, but they involve Bessel functions and Kloosterman sums.

\subsection{The Ramanujan Delta Function}

Recall that by definition $\Delta$ is the generator of the $1$-dimensional
space $S_{12}(\G)$ whose Fourier coefficient of $q^1$ is normalized to be
equal to $1$. By simple computation, we find the first terms in the Fourier
expansion of $\Delta$:
$$\Delta(\tau)=q-24q^2+252q^3-1472q^4+\cdots\;,$$
with no apparent formula for the coefficients. The $n$th coefficient is
denoted $\tau(n)$ (no confusion with $\tau\in\H$), and called Ramanujan's
tau function, and $\Delta$ itself is called Ramanujan's Delta function.

Of course, using
$\Delta=(E_4^3-E_6^2)/1728$ and expanding the powers, one can give a
complicated but explicit formula for $\tau(n)$ in terms
of the functions $\sigma_3$ and $\sigma_5$, but this is far from being the
best way to compute them. In fact, the following exercise already
gives a much better method.

\begin{exercise}\label{prob1} Let $D$ be the differential operator
$(1/(2\pi i))d/d\tau=qd/dq$.
\begin{enumerate}\item Show that the function $F=4E_4D(E_6)-6E_6D(E_4)$
  is a modular form of weight $12$, then by looking at its constant term show
  that it is a cusp form, and finally compute the constant $c$ such that
  $F=c\cdot\Delta$.
\item Deduce the formula
  $$\tau(n)=\dfrac{n}{12}(5\sigma_3(n)+7\sigma_5(n))+70\sum_{1\le m\le n-1}(2n-5m)\sigma_3(m)\sigma_5(n-m)\;.$$
\item Deduce in particular the congruences
  $\tau(n)\equiv n\sigma_5(n)\equiv n\sigma_1(n)\pmod5$ and
  $\tau(n)\equiv n\sigma_3(n)\pmod7$.
\end{enumerate}
\end{exercise}

Although there are much faster methods, this is already a very reasonable
way to compute $\tau(n)$.

The cusp form $\Delta$ is one of the most important functions in the theory of
modular forms. Its first main property, which is not at all apparent from
its definition, is that it has a \emph{product expansion}:

\begin{theorem} We have
  $$\Delta(\tau)=q\prod_{n\ge1}(1-q^n)^{24}\;.$$
\end{theorem}

\begin{proof} We are not going to give a complete proof, but sketch a method which is
one of the most natural to obtain the result.

We start backwards, from the product $R(\tau)$ on the right-hand side.
The logarithm transforms products into sums, but in the case of
\emph{functions} $f$, the \emph{logarithmic derivative} $f'/f$ (more precisely
$D(f)/f$, where $D=qd/dq$) also does this, and it is also more convenient.
We have
$$D(R)/R=1-24\sum_{n\ge1}\dfrac{nq^n}{1-q^n}=1-24\sum_{n\ge1}\sigma_1(n)q^n$$
as is easily seen by expanding $1/(1-q^n)$ as a geometric series. This
is exactly the case $k=2$ of the Eisenstein series $E_k$, which we have
excluded from our discussion for convergence reasons, so we come back to
our series $G_{2k}$ (we will divide by the normalizing factor
$2\z(2)=\pi^2/3$ at the end), and introduce a convergence factor due to Hecke,
setting
$$G_{2,s}(\tau)=\sum_{(m,n)\in\Z^2\setminus\{(0,0)\}}(m\tau+n)^{-2}|m\tau+n|^{-2s}\;.$$
As above this converges for $\Re(s)>0$, satisfies
$$G_{2,s}(\ga(\tau))=(c\tau+d)^2|c\tau+d|^{2s}G_{2,s}(\tau)$$
hence in particular is periodic of period $1$. It is straightforward to
compute its Fourier expansion, which we will not do here, and the Fourier
expansion shows that $G_{2,s}$ has an \emph{analytic continuation} to the
whole complex plane. In particular, the limit as $s\to0$ makes sense;
if we denote it by $G_2^*(\tau)$, by continuity it will of course satisfy
$G_2^*(\ga(\tau))=(c\tau+d)^2G_2^*(\tau)$, and the analytic continuation
of the Fourier expansion that has been computed gives
$$G_2^*(\tau)=\dfrac{\pi^2}{3}\left(1-\dfrac{3}{\pi\Im(\tau)}-24\sum_{n\ge1}\sigma_1(n)q^n\right)\;.$$
Note the essential fact that there is now a \emph{nonanalytic term}
$3/(\pi\Im(\tau))$. We will of course set the following definition:

\begin{definition} We define
  $$E_2(\tau)=1-24\sum_{n\ge1}\sigma_1(n)q^n\text{\quad and\quad}E_2^*(\tau)=E_2(\tau)-\dfrac{3}{\pi\Im(\tau)}\;.$$
\end{definition}

Thus $E_2(\tau)=D(R)/R$, $G_2^*(\tau)=(\pi^2/3)E_2^*(\tau)$, and we have
the following:

\begin{proposition}\label{prop:E2} For any $\ga=\psmm{a}{b}{c}{d}\in\G$ We have
  $E_2^*(\ga(\tau))=(c\tau+d)^2E_2^*(\tau)$. Equivalently,
  $$E_2(\ga(\tau))=(c\tau+d)^2E_2(\tau)+\dfrac{12}{2\pi i}c(c\tau+d)\;.$$
\end{proposition}

\begin{proof} The first result has been seen above, and the second follows from the
formula $\Im(\ga(\tau))=\Im(\tau)/|c\tau+d|^2$.\fp\end{proof}

\begin{exercise}\label{ex:e2} Show that
  $$E_2(\tau)=-24\left(-\dfrac{1}{24}+\sum_{m\ge1}\dfrac{m}{q^{-m}-1}\right)\;.$$
\end{exercise}

{\it Proof of the theorem. \/}
We can now prove the theorem on the product expansion of $\Delta$: noting that
$(d/d\tau)\ga(\tau)=1/(c\tau+d)^2$, the above formulas imply that
if we set $S=R(\ga(\tau))$ we have
\begin{align*}\dfrac{D(S)}{S}&=\dfrac{D(R)}{R}(\ga(\tau))(d/d\tau)(\ga(\tau))\\
  &=(c\tau+d)^{-2}E_2(\ga(\tau))=E_2(\tau)+\dfrac{12}{2\pi i}\dfrac{c}{c\tau+d}\\
  &=\dfrac{D(R)}{R}(\tau)+12\dfrac{D(c\tau+d)}{c\tau+d}\;.\end{align*}
By integrating and exponentiating, it follows that
$$R(\ga(\tau))=(c\tau+d)^{12}R(\tau)\;,$$
and since clearly $R$ is holomorphic on $\H$ and tends to $0$ as
$\Im(\tau)\to\infty$ (i.e., as $q\to0$), it follows that $R$ is a cusp
form of weight $12$ on $\G$, and since $S_{12}(\G)$ is $1$-dimensional and
the coefficient of $q^1$ in $R$ is $1$, we have $R=\Delta$, proving the
theorem.\fp\end{proof}

\begin{exercise} We have shown in passing that $D(\Delta)=E_2\Delta$.
Expanding the Fourier expansion of both sides, show that we have the recursion
$$(n-1)\tau(n)=-24\sum_{1\le m\le n-1}\sigma_1(m)\tau(n-m)\;.$$
\end{exercise}

\begin{exercise}
  \begin{enumerate}
  \item Let $F\in M_k(\G)$, and for some \emph{squarefree} integer $N$ set
    $$G(\tau)=\sum_{d\mid N}\mu(d)d^{k/2}F(d\tau)\;,$$
    where $\mu$ is the M\"obius function. Show that $G|_kW_N=\mu(N)G$, where
    $W_N=\psmm{0}{-1}{N}{0}$ is the so-called \emph{Fricke involution}.
  \item Show that if $N>1$ the same result is true for $F=E_2$, although $E_2$
    is only quasi-modular.
  \item Deduce that if $\mu(N)=(-1)^{k/2-1}$ we have $G(i/\sqrt{N})=0$.
  \item Applying this to $E_2$ and using Exercise \ref{ex:e2}, deduce that if
    $\mu(N)=1$ and $N>1$ we have
    $$\sum_{\gcd(m,N)=1}\dfrac{m}{e^{2\pi m/\sqrt{N}}-1}=\dfrac{\phi(N)}{24}\;,$$
    where $\phi(N)$ is Euler's totient function.
  \item Using directly the functional equation of $E_2^*$, show that for
    $N=1$ there is an additional term $-1/(8\pi)$, i.e., that
    $$\sum_{m\ge1}\dfrac{m}{e^{2\pi m}-1}=\dfrac{1}{24}-\dfrac{1}{8\pi}\;.$$  
  \end{enumerate}
\end{exercise}

\subsection{Product Expansions and the Dedekind Eta Function}

We continue our study of product expansions. We first mention an important
identity due to Jacobi, the triple product identity, as well as some
consequences:

\begin{theorem}[Triple product identity] If $|q|<1$ and $u\ne0$ we have
  $$\prod_{n\ge1}(1-q^n)(1-q^nu)\prod_{n\ge0}(1-q^n/u)
  =\sum_{k\ge0}(-1)^k(u^k-u^{-(k+1)})q^{k(k+1)/2}\;.$$
\end{theorem}

\begin{proof} (sketch): denote by $L(q,u)$ the left-hand side. We have clearly
$L(q,u/q)=-uL(q,u)$, and since one can write $L(q,u)=\sum_{k\in\Z}a_k(q)u^k$
this implies the recursion $a_k(q)=-q^ka_{k-1}(q)$, so
$a_k(q)=(-1)^kq^{k(k+1)/2}a_0(q)$, and separating $k\ge0$ and $k<0$ this
shows that
$$L(q,u)=a_0(q)\sum_{k\ge0}(-1)^k(u^k-u^{-(k+1)})q^{k(k+1)/2}\;.$$
The slightly longer part is to show that $a_0(q)=1$: this is done by
setting $u=i/q^{1/2}$ and $u=1/q^{1/2}$, which after a little computation
implies that $a(q^4)=a(q)$, and from there it is immediate to deduce that
$a(q)$ is a constant, and equal to $1$.\fp\end{proof}

To give the next corollaries, we need to define the
\emph{Dedekind eta function} $\eta(\tau)$, by
$$\eta(\tau)=q^{1/24}\prod_{n\ge1}(1-q^n)\;,$$
(recall that $q^{\al}=e^{2\pi i\al\tau}$).
Thus by definition $\eta(\tau)^{24}=\Delta(\tau)$. Since
$\Delta(-1/\tau)=\tau^{12}\Delta(\tau)$, it follows that
$\eta(-1/\tau)=c\cdot(\tau/i)^{1/2}\eta(\tau)$ for some $24$th root of unity
$c$ (where we always use the principal determination of the square root), and
since we see from the infinite product that $\eta(i)\ne0$, replacing
$\tau$ by $i$ shows that in fact $c=1$. Thus $\eta$ satisfies the two
basic modular equations
$$\eta(\tau+1)=e^{2\pi i/24}\eta(\tau)\text{\quad and\quad}\eta(-1/\tau)=(\tau/i)^{1/2}\eta(\tau)\;.$$
Of course we have more generally
$$\eta(\ga(\tau))=v_{\eta}(\ga)(c\tau+d)^{1/2}\eta(\tau)$$
for any $\ga\in\G$, with a complicated $24$th root of unity $v_{\eta}(\ga)$,
so $\eta$ is in some (reasonable) sense a modular form of weight $1/2$,
similar to the function $\th$ that we introduced at the very beginning.

\medskip

The triple product identity immediately implies the following two identities:

\begin{corollary}\label{coreta} We have
  \begin{align*}
    \eta(\tau)&=q^{1/24}\left(1+\sum_{k\ge1}(-1)^k(q^{k(3k-1)/2}+q^{k(3k+1)/2})\right)\text{\quad and}\\
    \eta(\tau)^3&=q^{1/8}\sum_{k\ge0}(-1)^k(2k+1)q^{k(k+1)/2}\;.
  \end{align*}
\end{corollary}

\begin{proof} In the triple product identity, replace $(u,q)$ by $(1/q,q^3)$: we
obtain
$$\prod_{n\ge1}(1-q^{3n})(1-q^{3n-1})\prod_{n\ge0}(1-q^{3n+1})=
\sum_{k\ge0}(-1)^k(q^{-k}-q^{k+1})q^{3k(k+1)/2}\;.$$
The left-hand side is clearly equal to $\eta(\tau)$, and the right-hand side
to
\begin{align*}&1-q+\sum_{k\ge1}(-1)^k(q^{k(3k+1)/2}-q^{(k+1)(3k+2)/2})\\
  &=1+\sum_{k\ge1}(-1)^kq^{k(3k+1)/2}-q+\sum_{k\ge2}(-1)^kq^{k(3k-1)/2}\;,
\end{align*}
giving the formula for $\eta(\tau)$. For the second formula, divide
the triple product identity by $1-1/u$ and make $u\to1$.\fp\end{proof}

Thus the first few terms are:
\begin{align*}
  \prod_{n\ge1}(1-q^n)&=1-q-q^2+q^5+q^7-q^{12}-q^{15}+\cdots\\
  \prod_{n\ge1}(1-q^n)^3&=1-3q+5q^3-7q^6+9q^{10}-11q^{15}+\cdots\;.
\end{align*}

The first identity was proved by L.~Euler.

\begin{exercise}\begin{enumerate}
\item Show that $24\Delta D(\eta)=\eta D(\Delta)$, and using the explicit
  Fourier expansion of $\eta$, deduce the recursion
  $$\sum_{k\in\Z}(-1)^k(75k^2+25k+2-2n)\tau\left(n-\dfrac{k(3k+1)}{2}\right)=0\;.$$
\item Similarly, from $8\Delta D(\eta^3)=\eta^3 D(\Delta)$ deduce the recursion
  $$\sum_{k\in\Z}(-1)^k(2k+1)(9k^2+9k+2-2n)\tau\left(n-\dfrac{k(k+1)}{2}\right)=0\;.$$
\end{enumerate}
\end{exercise}

\begin{exercise}\label{expoch} Define the \emph{$q$-Pochhammer symbol} $(q)_n$
by $(q)_n=(1-q)(1-q^2)\cdots(1-q^n)$.
\begin{enumerate}\item Set $f(a,q)=\prod_{n\ge1}(1-aq^n)$,
and define coefficients $c_n(q)$ by setting $f(a,q)=\sum_{n\ge0}c_n(q)a^n$.
Show that $f(a,q)=(1-aq)f(aq,q)$, deduce that $c_n(q)(1-q^n)=-q^nc_{n-1}(q)$
and finally the identity
$$\prod_{n\ge1}(1-aq^n)=\sum_{n\ge0}(-1)^na^nq^{n(n+1)/2}/(q)_n\;.$$
\item Write in terms of the Dedekind eta function the identities obtained
  by specializing to $a=1$, $a=-1$, $a=-1/q$, $a=q^{1/2}$, and $a=-q^{1/2}$.
\item Similarly, prove the identity
$$1/\prod_{n\ge1}(1-aq^n)=\sum_{n\ge0}a^nq^n/(q)_n\;,$$
  and once again write in terms of the Dedekind eta function the identities
  obtained by specializing to the same five values of $a$.
\item By multiplying two of the above identities and using the triple product
  identity, prove the identity
  $$\dfrac{1}{\prod_{n\ge1}(1-q^n)}=\sum_{n\ge0}\dfrac{q^{n^2}}{(q)_n^2}\;.$$
\end{enumerate}\end{exercise}

Note that this last series is the generating function of the \emph{partition
function $p(n)$}, so if one wants to make a table of $p(n)$ up to $n=10000$,
say, using the left-hand side would require $10000$ terms, while using the
right-hand side only requires $100$.

\subsection{Computational Aspects of the Ramanujan $\tau$ Function}

Since its introduction, the Ramanujan tau function $\tau(n)$ has fascinated
number theorists. For instance there is a conjecture due to D.~H.~Lehmer
that $\tau(n)\ne0$, and an even stronger conjecture (which would imply
the former) that for every prime $p$ we have $p\nmid\tau(p)$
(on probabilistic grounds, the latter conjecture is probably false).

To test these conjectures as well as others, it is an interesting
computational challenge to \emph{compute} $\tau(n)$ for large $n$
(because of Ramanujan's first two conjectures, i.e., Mordell's theorem
that we will prove in Section \ref{sec:ram} below, it is sufficient to compute
$\tau(p)$ for $p$ \emph{prime}).

We can have two distinct goals. The first is to compute a \emph{table} of
$\tau(n)$ for $n\le B$, where $B$ is some (large) bound. The second
is to compute \emph{individual values} of $\tau(n)$, equivalently of
$\tau(p)$ for $p$ prime.

\medskip

Consider first the construction of a \emph{table}. The use of the first
recursion given in the above exercise needs $O(n^{1/2})$ operations
per value of $\tau(n)$, hence $O(B^{3/2})$ operations in all to have a
table for $n\le B$.

However, it is well known that the \emph{Fast Fourier Transform} (FFT)
allows one to compute products of power series in essentially linear time.
Thus, using Corollary \ref{coreta}, we can directly write the power series
expansion of $\eta^3$, and use the FFT to compute its eighth power
$\eta^{24}=\Delta$. This will require $O(B\log(B))$ operations,
so is much faster than the preceding method; it is essentially optimal since
one needs $O(B)$ time simply to write the result.

\medskip

Using large computer resources, especially in memory, it is reasonable to
construct a table up to $B=10^{12}$, but not much more. Thus, the problem
of computing \emph{individual} values of $\tau(p)$ is important.
We have already seen one such method in Exercise \ref{prob1} above, which
gives a method for computing $\tau(n)$ in time $O(n^{1+\eps})$ for any
$\eps>0$.

A deep and important theorem of B.~Edixhoven, J.-M.~Couveignes, et al.,
says that it is possible to compute $\tau(p)$ in time \emph{polynomial}
in $\log(p)$, and in particular in time $O(p^{\eps})$ for any $\eps>0$.
Unfortunately this algorithm is not at all practical, and at least
for now, completely useless for us. The only practical and important
application is for the computation of $\tau(p)$ modulo some small prime
numbers $\ell$ (typically $\ell<50$, so far from being sufficient to apply
the Chinese Remainder Theorem).

However, there exists an algorithm which takes time $O(n^{1/2+\eps})$
for any $\eps>0$, so much better than the one of Exercise \ref{prob1}, and
which is very practical. It is based on the use of the Eichler--Selberg
\emph{trace formula}, together with the computation of \emph{Hurwitz
class numbers} $H(N)$ (essentially the class numbers of imaginary quadratic
orders counted with suitable multiplicity): if we set $H_3(N)=H(4N)+2H(N)$
(note that $H(4N)$ can be computed in terms of $H(N)$), then for $p$ prime
\begin{align*}\tau(p)&=28p^6-28p^5-90p^4-35p^3-1\\
  &\phantom{=}-128\sum_{1\le t<p^{1/2}}t^6(4t^4-9pt^2+7p^2)H_3(p-t^2)\;.
\end{align*}
See \cite{Coh-Str} Exercise 12.13 of Chapter 12 for details.
Using this formula and a cluster, it should be reasonable to compute $\tau(p)$
for $p$ of the order of $10^{16}$.

\subsection{Modular Functions and Complex Multiplication}

Although the terminology is quite unfortunate, we cannot change it. By
definition, a modular \emph{function} is a function $F$ from $\H$ to $\C$
which is weakly modular of weight $0$ (so that $F(\ga(\tau))=F(\tau)$, in
other words is \emph{invariant} under $\G$, or equivalently defines a function
from $\G\backslash\H$ to $\C$), meromorphic, including at $\infty$. This
last statement requires some additional explanation, but in simple terms,
this means that the Fourier expansion of $F$ has only finitely many Fourier
coefficients for negative powers of $q$: $F(\tau)=\sum_{n\ge n_0}a(n)q^n$,
for some (possibly negative) $n_0$.

A trivial way to obtain modular functions is simply to take the quotient
of two modular forms having the same weight. The most important is the
$j$-function defined by
$$j(\tau)=\dfrac{E_4^3(\tau)}{\Delta(\tau)}\;,$$
whose Fourier expansion begins by
$$j(\tau)=\dfrac{1}{q}+744+196884q+21493760q^2+\cdots$$
Indeed, one can easily prove the following theorem:

\begin{theorem}\label{thmratj} Let $F$ be a meromorphic function on $\H$. The
  following are equivalent:
  \begin{enumerate}\item $F$ is a modular function.
  \item $F$ is the quotient of two modular forms of equal weight.
  \item $F$ is a rational function of $j$.
  \end{enumerate}
\end{theorem}

\begin{exercise}\label{exj} \begin{enumerate}
\item Noting that Theorem \ref{thmval} is valid more generally for modular
  functions (with $v_{\tau}(f)=-r<0$ if $f$ has a pole of order $r$ at $\tau$)
  and using the specific properties of $j(\tau)$, compute
  $v_{\tau}(f)$ for the functions $j(\tau)$, $j(\tau)-1728$, and $D(j)(\tau)$,
  at the points $\rho=e^{2\pi i/3}$, $i$, $i\infty$, and $\tau_0$ for $\tau_0$
  distinct from these three special points.
\item Set $f=f(a,b,c)=D(j)^a/(j^b(j-1728)^c)$. Show that $f$ is a modular
  \emph{form} if and only if $2c\le a$, $3b\le 2a$, and $b+c\ge a$, and give
  similar conditions for $f$ to be a \emph{cusp form}.
\item Show that $E_4=f(2,1,1)$, $E_6=f(3,2,1)$, and $\Delta=f(6,4,3)$, so
  that for instance $D(j)=-E_{14}=-E_4^2E_6/\Delta$.
\end{enumerate}
\end{exercise}

An important theory linked to modular functions is the theory of
\emph{complex multiplication}, which deserves a course in itself. We simply
mention one of the basic results.

We will say that a complex number $\tau\in\H$ is a CM point (CM for Complex
Multiplication) if it belongs to an imaginary quadatic field, or equivalently
if there exist integers $a$, $b$, and $c$ with $a\ne0$ such that
$a\tau^2+b\tau+c=0$. The first basic theorem is the following:

\begin{theorem} If $\tau$ is a CM point then $j(\tau)$ is an algebraic
  integer.\end{theorem}

Note that this theorem has two parts: the first and most important part is
that $j(\tau)$ is algebraic. This is in fact easy to prove. The second part is
that it is an algebraic \emph{integer}, and this is more difficult.
Since any modular function $f$ is a rational function of $j$, it follows that
if this rational function has algebraic coefficients then $f(\tau)$ will be
algebraic (but not necessarily integral). Another immediate consequence is
the following:

\begin{corollary} Let $\tau$ be a CM point and define $\Om_\tau=\eta(\tau)^2$,
  where $\eta$ is as usual the Dedekind eta function. For any modular form
  $f$ of weight $k$ (in fact $f$ can also be meromorphic) the number
  $f(\tau)/\Om_\tau^k$ is algebraic. In fact $E_4(\tau)/\Om_\tau^4$ and
  $E_6(\tau)/\Om_\tau^6$ are always algebraic \emph{integers}.
  \end{corollary}

But the importance of this theorem lies in algebraic number theory. We give the
following theorem without explaining the necessary notions:

\begin{theorem} Let $\tau$ be a CM point, and $D=b^2-4ac$ its
  \emph{discriminant}, where we choose $\gcd(a,b,c)=1$. Then $\Q(j(\tau))$
  is the \emph{ring class field} of discriminant $D$, and in particular if
  $D$ is the discriminant of a quadratic field $K=\Q(\sqrt{D})$, then
  $K(j(\tau))$ is the \emph{Hilbert class field} of $K$. In particular,
  the degree of the minimal polynomial of the algebraic integer $j(\tau)$ is
  equal to the \emph{class number} $h(D)$ of the order of discriminant $D$.
\end{theorem}

Examples:

\begin{align*}
j((1+i\sqrt3)/2)&=0=1728-3(24)^2\\ 
j(i)&=1728=12^3=1728-4(0)^2\\
j((1+i\sqrt7)/2)&=-3375=(-15)^3=1728-7(27)^2\\
j(i\sqrt2)&=8000=20^3=1728+8(28)^2\\
j((1+i\sqrt{11})/2)&=-32768=(-32)^3=1728-11(56)^2\\
j((1+i\sqrt{163})/2)&=-262537412640768000=(-640320)^3\\
&=1728-163(40133016)^2\\
j(i\sqrt3)&=54000=2(30)^3=1728+12(66)^2\\
j(2i)&=287496=(66)^3=1728+8(189)^2\\
j((1+3i\sqrt3)/2)&=-12288000=-3(160)^3=1728-3(2024)^2\\
j((1+i\sqrt{15})/2)&=\dfrac{-191025-85995\sqrt5}{2}\\
&=\dfrac{1-\sqrt5}{2}\left(\dfrac{75+27\sqrt5}{2}\right)^3=
1728-3\left(\dfrac{273+105\sqrt5}{2}\right)^2
\end{align*}

Note that we give the results in the above form since it can be shown that
the functions $j^{1/3}$ and $(j-1728)^{1/2}$ also have interesting arithmetic
properties.

The example with $D=-163$ is particularly spectacular:

\begin{exercise} Using the above table, show that
  $$(e^{\pi\sqrt{163}}-744)^{1/3}=640320-\eps\;,$$
  with $0<\eps<10^{-24}$, and more precisely that $\eps$ is approximately
  equal to $65628e^{-(5/3)\pi\sqrt{163}}$ (note that $65628=196884/3$).
\end{exercise}

\begin{exercise}\begin{enumerate}
  \item Using once again the example of $163$, compute heuristically
  a few terms of the Fourier expansion of $j$ assuming that it is of the
  form $1/q+\sum_{n\ge0}c(n)q^n$ with $c(n)$ reasonably small integers
  using the following method. Set $q=-e^{-\pi\sqrt{163}}$, and let
  $J=(-640320)^3$ be the exact value of $j((-1+i\sqrt{163})/2)$.
  By computing $J-1/q$, one notices that the result is very close to $744$,
  so we guess that $c(0)=744$. We then compute $(J-1/q-c(0))/q$ and note
  that once again the result is close to an integer, giving $c(1)$, and so
  on. Go as far as you can with this method.
\item Do the same for $67$ instead of $163$. You will find the same
  Fourier coefficients (but you can go less far).
\item On the other hand, do the same for $58$, starting with $J$ equal to
  the integer close to $e^{\pi\sqrt{58}}$. You will find a \emph{different}
  Fourier expansion: it corresponds in fact to another modular function,
  this time defined on a subgroup of $\G$, called a \emph{Hauptmodul}.
\item Try to find other rational numbers $D$ such that $e^{\pi\sqrt{D}}$
  is close to an integer, and do the same exercise for them (an example
  where $D$ is not integral is $89/3$).
\end{enumerate}\end{exercise}

\subsection{Derivatives of Modular Forms}\label{sec:deriv}

If we differentiate the modular equation
$f((a\tau+b)/(c\tau+d))=(c\tau+d)^kf(\tau)$ with $\psmm{a}{b}{c}{d}\in\G$
using the operator $D=(1/(2\pi i))d/d\tau$ (which gives simpler formulas than
$d/d\tau$ since $D(q^n)=nq^n$), we easily obtain
$$D(f)\left(\dfrac{a\tau+b}{c\tau+d}\right)=(c\tau+d)^{k+2}\left(D(f)(\tau)+\dfrac{k}{2\pi i}\dfrac{c}{c\tau+d}f(\tau)\right)\;.$$
Thus the derivative of a weakly modular form of weight $k$ looks like one
of weight $k+2$, except that there is an extra term. This term vanishes if
$k=0$, so the derivative of a modular function of weight $0$ is indeed
modular of weight $2$ (we have seen above the example of $j(\tau)$ which
satisfies $D(j)=-E_{14}/\Delta$).

If $k>0$ and we really want a true weakly modular form of weight $k+2$
there are two ways to do this. The first one is called the
\emph{Serre derivative}:

\begin{exercise} Using Proposition \ref{prop:E2}, show that if $f$ is
  weakly modular of weight $k$ then $D(f)-(k/12)E_2f$ is weakly modular
  of weight $k+2$. In particular, if $f\in M_k(\G)$ then
  $SD_k(f):=D(f)-(k/12)E_2f\in M_{k+2}(\G)$.\end{exercise}

The second method is to set $D^*(f):=D(f)-(k/(4\pi\Im(\tau)))f$ since by
Proposition \ref{prop:E2} we have $D^*(f)=SD_k(f)-(k/12)E_2^*f$. This loses
holomorphy, but is very useful in certain contexts.

Note that if more than one modular form is involved, there are more ways
to make new modular forms using derivatives:

\begin{exercise}\label{ex:RC}\begin{enumerate}
\item For $i=1$, $2$ let $f_i\in M_{k_i}(\G)$. By considering the
  modular function $f_1^{k_2}/f_2^{k_1}$ of weight $0$, show that
  $$k_2f_2D(f_1)-k_1f_1D(f_2)\in S_{k_1+k_2+2}(\G)\;.$$
  Note that this generalizes Exercise \ref{prob1}.
\item Compute constants $a$, $b$, and $c$ (depending on $k_1$ and $k_2$ and
  not all $0$) such that
  $$[f_1,f_2]_2=aD^2(f_1)+bD(f_1)D(f_2)+cD^2(f_2)\in S_{k_1+k_2+4}(\G)\;.$$
\end{enumerate}\end{exercise}

This gives the first two of the so-called \emph{Rankin--Cohen} brackets.

\smallskip

As an application of derivatives of modular forms, we give a proof of a
theorem of Siegel. We begin by the following:

\begin{lemma} Let $a$ and $b$ be nonnegative integers such that $4a+6b=12r+2$.
The constant term of the Fourier expansion of $F_r(a,b)=E_4^aE_6^b/\Delta^r$
vanishes.
\end{lemma}

\begin{proof} By assumption $F_r(a,b)$ is a meromorphic modular form of weight
$2$. Since $D(\sum_{n\ge n_0}a(n)q^n)=\sum_{n\ge n_0}na(n)q^n$, it is
sufficient to find a modular function $G_r(a,b)$ of weight $0$ such that
$F_r(a,b)=D(G_r(a,b))$ (recall that the derivative of a modular function of
weight $0$ is still modular). We prove this by an induction first on $r$, then
on $b$. Recall that by Exercise \ref{exj} we have
$D(j)=-E_{14}/\Delta=-E_4^2E_6/\Delta$, and since $4a+6b=14$ has only the
solution $(a,b)=(2,1)$ the result is true for $r=1$. Assume it is true for
$r-1$. We now do a recursion on $b$, noting that since $2a+3b=6r+1$, $b$ is
odd. Note that $D(j^r)=rj^{r-1}D(j)=-rE_4^{3r-1}E_6/\Delta^r$, so the constant
term of $F_r(a,1)$ indeed vanishes. However, since $E_4^3-E_6^2=1728\Delta$,
if $a\ge3$ we have
$$F_r(a-3,b+2)=E_4^{a-3}E_6^b(E_4^3-1728\Delta)/\Delta^r
=F_r(a,b)-1728F_{r-1}(a-3,b)\;,$$
proving that the result is true for $r$ by induction on $b$ since we assumed
it true for $r-1$.\fp\end{proof}

We can now prove (part of) Siegel's theorem:

\begin{theorem}\label{thmsieg} For $r=\dim(M_k(\G))$ define coefficients
  $c_i^k$ by $$\dfrac{E_{12r-k+2}}{\Delta^r}=\sum_{i\ge-r}c_i^kq^i\;,$$
  where by convention we set $E_0=1$. Then for any
  $f=\sum_{n\ge0}a(n)\in M_k(\G)$ we have the relation
  $$\sum_{0\le n\le r}c_{-n}^ka(n)=0\;.$$
  In addition we have $c_0^k\ne0$, so that
  $a(0)=\sum_{1\le n\le r}(c_{-n}^k/c_0^k)a(n)$ is a linear combination
  with \emph{rational coefficients} of the $a(n)$ for $1\le n\le r$.
\end{theorem}

\begin{proof} First note that by Corollary \ref{cordim} we have
  $r\ge(k-2)/12$ (with equality only if $k\equiv2\pmod{12}$), so the
  definition of the coefficients $c_i^k$ makes sense.
  Note also that since the Fourier expansion of $E_{12r-k+2}$ begins with
  $1+O(q)$ and that of $\Delta^r$ by $q^r+O(q^{r+1})$, that of the
  quotient begins with $q^{-r}+O(q^{1-r})$ (in particular $c_{-r}^k=1$).
  The proof of the first part is now immediate: the modular form
  $fE_{12r-k+2}$ belongs
  to $M_{12r+2}(\G)$, so by Corollary \ref{core4e6} is a linear combination
  of $E_4^aE_6^b$ with $4a+6b=12r+2$. It follows from the lemma that
  the constant term of $fE_{12r-k+2}/\Delta^r$ vanishes, and this
  constant term is equal to $\sum_{0\le n\le r}c_{-n}^ka(n)$, proving
  the first part of the theorem. The fact that $c_0^k\ne0$ (which is of
  course essential) is a little more difficult and will be omitted, see
  \cite{Coh-Str} Theorem 9.5.1.\fp\end{proof}

  This theorem has (at least) two consequences. First, a theoretical one:
  if one can construct a modular form whose constant term is some interesting
  quantity and whose Fourier coefficients $a(n)$ are rational, this shows
  that the interesting quantity is also rational. This is what allowed
  Siegel to show that the value at negative integers of Dedekind zeta functions
  of totally real number fields are rational, see Section \ref{sec:several}.
  Second, a practical one: it allows to compute explicitly the constant
  coefficient $a(0)$ in terms of the $a(n)$, giving interesting formulas,
  see again Section \ref{sec:several}.
  
\section{Hecke Operators: Ramanujan's discoveries}\label{sec:ram}

We now come to one of the most amazing and important discoveries on modular
forms due to S.~Ramanujan, which has led to the modern development of the
subject. Recall that we set
$$\Delta(\tau)=q\prod_{m\ge1}(1-q^m)^{24}=\sum_{n\ge1}\tau(n)q^n\;.$$
We have $\tau(2)=-24$, $\tau(3)=252$, and $\tau(6)=-6048=-24\cdot252$, so
that $\tau(6)=\tau(2)\tau(3)$. After some more experiments, Ramanujan
conjectured that if $m$ and $n$ are coprime we have $\tau(mn)=\tau(m)\tau(n)$.
Thus, by decomposing an integer into products of prime powers, assuming this
conjecture, we are reduced to the study of $\tau(p^k)$ for $p$ prime.

Ramanujan then noticed that $\tau(4)=-1472=(-24)^2-2^{11}=\tau(2)^2-2^{11}$,
and again after some experiments he conjectured that
$\tau(p^2)=\tau(p)^2-p^{11}$, and more generally that
$\tau(p^{k+1})=\tau(p)\tau(p^k)-p^{11}\tau(p^{k-1})$.
Thus $u_k=\tau(p^k)$ satisfies a linear recurrence relation
$$u_{k+1}-\tau(p)u_k+p^{11}u_{k-1}=0\;,$$ and since $u_0=1$ the sequence
is entirely determined by the value of $u_1=\tau(p)$. It is well-known that
the behavior of a linear recurrent sequence is determined by its
\emph{characteristic polynomial}. Here it is equal to $X^2-\tau(p)X+p^{11}$,
and the third of Ramanujan's conjectures is that the discriminant of this
equation is always negative, or equivalently that $|\tau(p)|<p^{11/2}$.

Note that if $\al_p$ and $\be_p$ are the roots of the characteristic
polynomial (necessarily distinct since we cannot have $|\tau(p)|=p^{11/2}$),
then $\tau(p^k)=(\al_p^{k+1}-\be_p^{k+1})/(\al_p-\be_p)$, and the last
conjecture says that $\al_p$ and $\be_p$ are \emph{complex conjugate},
and in particular of modulus \emph{equal} to $p^{11/2}$.

These conjectures are all true. The first two (multiplicativity and
recursion) were proved by L.~Mordell only one year after Ramanujan formulated
them, and indeed the proof is quite easy (in fact we will prove them below).
The third conjecture $|\tau(p)|<p^{11/2}$ is extremely hard, and was only
proved by P.~Deligne in 1970 using the whole machinery developed by the school
of A.~Grothendieck to solve the Weil conjectures .

The main idea of Mordell, which was generalized later by E.~Hecke, is to
introduce certain linear operators (now called Hecke operators) on spaces of
modular forms, to prove that they satisfy the multiplicativity and recursion
properties (this is in general much easier than to prove this on numbers),
and finally to use the fact that $S_{12}(\G)=\C\Delta$ is of dimension $1$, so
that necessarily $\Delta$ is an \emph{eigenform} of the Hecke operators
whose eigenvalues are exactly its Fourier coefficients.

\medskip

Although there are more natural ways of introducing them, we will define
the Hecke operator $T(n)$ on $M_k(\G)$ directly by its action on Fourier
expansions
$T(n)(\sum_{m\ge0}a(m)q^m)=\sum_{m\ge0}b(m)q^m$, where
$$b(m)=\sum_{d\mid\gcd(m,n)}d^{k-1}a(mn/d^2)\;.$$
Note that we can consider this definition as purely formal, apart from the
presence of the integer $k$ this is totally unrelated to the possible fact
that $\sum_{m\ge0}a(m)q^m\in M_k(\G)$.

A simple but slightly tedious combinatorial argument shows that these
operators satisfy
$$T(n)T(m)=\sum_{d\mid\gcd(n,m)}d^{k-1}T(nm/d^2)\;.$$
In particular if $m$ and $n$ are coprime we have $T(n)T(m)=T(nm)$
(multiplicativity), and if $p$ is a prime and $k\ge1$ we have
$T(p^k)T(p)=T(p^{k+1})+p^{k-1}T(p^{k-1})$ (recursion). This shows that
these operators are indeed good candidates for proving the first two of
Ramanujan's conjectures.

We need to show the essential fact that they preserve $M_k(\G)$ and
$S_k(\G)$ (the latter will follow from the former since by the above definition
$b(0)=\sum_{d\mid n}d^{k-1}a(0)=a(0)\sigma_{k-1}(n)=0$ if $a(0)=0$).
By recursion and multiplicativity, it is sufficient to show this for $T(p)$
with $p$ prime. Now if $F(\tau)=\sum_{m\ge0}a(m)q^m$,
$T(p)(F)(\tau)=\sum_{m\ge0}b(m)q^m$ with $b(m)=a(mp)$ if $p\nmid m$, and
$b(m)=a(mp)+p^{k-1}a(m/p)$ if $p\mid m$.

On the other hand, let us compute
$G(\tau)=\sum_{0\le j<p}F((\tau+j)/p)$. Replacing directly in the Fourier
expansion we have
$$G(\tau)=\sum_{m\ge0}a(m)q^{m/p}\sum_{0\le j<p}e^{2\pi imj/p}\;.$$
The inner sum is a complete geometric sum which vanishes unless $p\mid m$,
in which case it is equal to $p$. Thus, changing $m$ into $pm$ we have
$G(\tau)=p\sum_{m\ge0}a(pm)q^m$. On the other hand, we have trivially
$\sum_{p\mid m}a(m/p)q^m=\sum_{m\ge0}a(m)q^{pm}=F(p\tau)$.
Replacing both of these formulas in the formula for $T(p)(F)$ we see that
$$T(p)(F)(\tau)=p^{k-1}F(p\tau)+\dfrac{1}{p}\sum_{0\le j<p}F\left(\dfrac{\tau+j}{p}\right)\;.$$

\begin{exercise} Show more generally that
$$T(n)(F)(\tau)=\sum_{ad=n}a^{k-1}\dfrac{1}{d}\sum_{0\le b<d}F\left(\dfrac{a\tau+b}{d}\right)\;.$$\end{exercise}

It is now easy to show that $T(p)F$ is modular: replace $\tau$ by $\ga(\tau)$
in the above formula and make a number of elementary manipulations to prove
modularity. In fact, since $\G$ is generated by $\tau\mapsto\tau+1$ and
$\tau\mapsto-1/\tau$, it is immediate to check modularity for these two maps
on the above formula.

As mentioned above, the proof of the first two Ramanujan conjectures is
now immediate: since $T(n)$ acts on the one-dimensional space $S_{12}(\G)$
we must have $T(n)(\Delta)=c\cdot\Delta$ for some constant $c$. Replacing
in the definition of $T(n)$, we thus have for all $m$
$c\tau(m)=\sum_{d\mid \gcd(n,m)}d^{11}\tau(nm/d^2)$. Choosing $m=1$ and
using $\tau(1)=1$ shows that $c=\tau(n)$, so that
$$\tau(n)\tau(m)=\sum_{d\mid\gcd(n,m)}d^{11}\tau(nm/d^2)$$
which implies (and is equivalent to) the first two conjectures of
Ramanujan.

\smallskip

Denote by $P_k(n)$ the \emph{characteristic polynomial} of the linear map
$T(n)$ on $S_k(\G)$. A strong form of the so-called Maeda's conjecture
states that for $n>1$ the polynomial $P_k(n)$ is \emph{irreducible}. This
has been tested up to very large weights.

\begin{exercise} The above proof shows that the Hecke operators also
preserve the space of modular \emph{functions}, so by Theorem \ref{thmratj}
the image of $j(\tau)$ will be a rational function in $j$:
\begin{enumerate}\item Show for instance that
\begin{align*}T(2)(j)&=j^2/2-744j+81000\text{\quad and}\\
  T(3)(j)&=j^3/3-744j^2+356652j-12288000\;.\end{align*}
\item Set $J=j-744$, i.e., $j$ with no term in $q^0$ in its Fourier
  expansion. Deduce that
  \begin{align*}T(2)(J)&=J^2/2-196884\text{\quad and}\\
      T(3)(J)&=J^3/3-196884J-21493760\;,\end{align*}
  and observe that the coefficients that we obtain are exactly the Fourier
  coefficients of $J$.
\item Prove that $T(n)(j)$ is a \emph{polynomial} in $j$. Does the last
  observation generalize?
\end{enumerate}
\end{exercise}

\section{Euler Products, Functional Equations}

\subsection{Euler Products}

The case of $\Delta$ is quite special, in that the modular form space
to which it naturally belongs, $S_{12}(\G)$, is only $1$-dimensional.
As can easily be seen from the dimension formula, this occurs (for cusp forms)
only for $k=12$, $16$, $18$, $20$, $22$, and $26$ (there are no nonzero
cusp forms in weight $14$ and the space is of dimension $2$ in weight $24$),
and thus the evident cusp forms $\Delta E_{k-12}$ for these values of $k$
(setting $E_0=1$) are generators of the space $S_k(\G)$, so are eigenforms
of the Hecke operators and share exactly the same properties as $\Delta$, with
$p^{11}$ replaced by $p^{k-1}$.

When the dimension is greater than $1$, we must work slightly more. From
the formulas given above it is clear that the $T(n)$ form a \emph{commutative
  algebra} of operators on the finite dimensional vector space $S_k(\G)$.
In addition, we have seen above that there is a natural \emph{scalar product}
on $S_k(\G)$. One can show the not completely trivial fact that $T(n)$ is
Hermitian for this scalar product, hence in particular is diagonalizable.
It follows by an easy and classical result of linear algebra that these
operators are \emph{simultaneously diagonalizable}, i.e., there exists a basis
$F_i$ of forms in $S_k(\G)$ such that $T(n)F_i=\lambda_i(n)F_i$ for all $n$
and $i$. Identifying Fourier coefficients as we have done above for $\Delta$
shows that if $F_i=\sum_{n\ge1}a_i(n)q^n$ we have
$a_i(n)=\lambda_i(n)a_i(0)$. This implies first that $a_i(0)\ne0$, otherwise
$F_i$ would be identically zero, so that by dividing by $a_i(0)$ we can
always \emph{normalize} the eigenforms so that $a_i(0)=1$, and second,
as for $\Delta$, that $a_i(n)=\lambda_i(n)$, i.e., the eigenvalues are
exactly the Fourier coefficients. In addition, since the $T(n)$ are 
Hermitian, these eigenvalues are real for any embedding into $\C$, hence
are \emph{totally real}, in other words their minimal polynomial has only
real roots. Finally, using Theorem \ref{thmval}, it is immediate to show
that the field generated by the $a_i(n)$ is finite-dimensional over $\Q$, i.e.,
is a number field.

\begin{exercise} Consider the space $S=S_{24}(\G)$, which is the smallest
weight where the dimension is greater than $1$, here $2$. By the structure
theorem given above, it is generated for instance by $\Delta^2$ and
$\Delta E_4^3$. Compute the matrix of the operator $T(2)$ on this basis of $S$,
diagonalize this matrix, so find the \emph{eigenfunctions} of $T(2)$ on $S$
(the prime number $144169$ should occur). Check that these eigenfunctions
are also eigenfunctions of $T(3)$.
\end{exercise}

Thus, let $F=\sum_{n\ge1}a(n)q^n$ be a \emph{normalized} eigenfunction for all
the Hecke operators in $S_k(\G)$ (for instance $F=\Delta$ with $k=12$), and
consider the \emph{Dirichlet series}
$$L(F,s)=\sum_{n\ge1}\dfrac{a(n)}{n^s}\;,$$
for the moment formally, although we will show below that it converges
for $\Re(s)$ sufficiently large. The multiplicativity property of the
coefficients ($a(nm)=a(n)a(m)$ if $\gcd(n,m)=1$, coming from that of the
$T(n)$) is \emph{equivalent} to the fact that we have an \emph{Euler product}
(a product over primes)
$$L(F,s)=\prod_{p\in P}L_p(F,s)\text{\quad with\quad}L_p(F,s)=\sum_{j\ge0}\dfrac{a(p^j)}{p^{js}}\;,$$
where we will always denote by $P$ the set of prime numbers.

The additional recursion property $a(p^{j+1})=a(p)a(p^j)-p^{k-1}a(p^{j-1})$
is equivalent to the identity
$$L_p(F,s)=\dfrac{1}{1-a(p)p^{-s}+p^{k-1}p^{-2s}}$$
(multiply both sides by the denominator to check this). We have
thus proved the following theorem:

\begin{theorem} Let $F=\sum_{n\ge1}a(n)q^n\in S_k(\G)$ be an eigenfunction of
  all Hecke operators. We have an Euler product
  $$L(F,s)=\sum_{n\ge1}\dfrac{a(n)}{n^s}=\prod_{p\in P}\dfrac{1}{1-a(p)p^{-s}+p^{k-1}p^{-2s}}\;.$$
\end{theorem}

Note that we have not really used the fact that $F$ is a cusp form: the
above theorem is still valid if $F=F_k$ is the normalized Eisenstein series
$$F_k(\tau)=-\dfrac{B_k}{2k}E_k(\tau)=-\dfrac{B_k}{2k}+\sum_{n\ge1}\sigma_{k-1}(n)q^n\;,$$
which is easily seen to be a normalized eigenfunction for all Hecke operators.
In fact:

\begin{exercise} Let $a\in\C$ be any complex number and let as usual
$\sigma_a(n)=\sum_{d\mid n}d^a$.
\begin{enumerate}\item Show that
  $$\sum_{n\ge1}\dfrac{\sigma_a(n)}{n^s}=\z(s-a)\z(s)=\prod_{p\in P}\dfrac{1}{1-\sigma_a(p)p^{-s}+p^ap^{-2s}}\;,$$
  with $\sigma_a(p)=p^a+1$.
\item Show that
  $$\sigma_a(m)\sigma_a(n)=\sum_{d\mid\gcd(m,n)}d^a\sigma_a\left(\dfrac{mn}{d^2}\right)\;,$$
  so that in particular $F_k$ is indeed a normalized eigenfunction for all
  Hecke operators.
\end{enumerate}
\end{exercise}

\subsection{Analytic Properties of $L$-Functions}

Everything that we have done up to now is purely formal, i.e., we do not need
to assume convergence. However in the sequel we will need to prove some
analytic results, and for this we need to prove convergence for certain values
of $s$. We begin with the following easy bound, due to Hecke:

\begin{proposition} Let $F=\sum_{n\ge1}a(n)q^n\in S_k(\G)$ be a cusp form
  (not necessarily an eigenform). There exists a constant $c>0$ (depending on
  $F$) such that for all $n$ we have $|a(n)|\le c n^{k/2}$.\end{proposition}

\begin{proof} The trick is to consider the function
$g(\tau)=|F(\tau)\Im(\tau)^{k/2}|$: since
we have seen that $\Im(\ga(\tau))=\Im(\tau)/|c\tau+d|^2$, it follows that
$g(\tau)$ is \emph{invariant} under $\G$. It follows that
$\sup_{\tau\in\H}g(\tau)=\sup_{\tau\in\FF}g(\tau)$, where $\FF$ is the
fundamental domain used above. Now because of the Fourier expansion and
the fact that $F$ is a cusp form, $|F(\tau)|=O(e^{-2\pi\Im(\tau)})$ as
$\Im(\tau)\to\infty$, so $g(\tau)$ tends to $0$ also. It immediately follows
that $g$ is \emph{bounded} on $\FF$, hence on $\H$, so that there exists
a constant $c_1>0$ such that $|F(\tau)|\le c_1\Im(\tau)^{-k/2}$ for all $\tau$.

We can now easily prove Hecke's bound: from the Fourier series section we know
that for any $y>0$
$$a(n)=e^{2\pi ny}\int_0^1 F(x+iy)e^{-2\pi inx}\,dx\;,$$
so that $|a(n)|\le c_1e^{2\pi ny}y^{-k/2}$, and choosing $y=1/n$ proves the
proposition with $c=e^{2\pi}c_1$.\fp\end{proof}

The following corollary is now clear:

\begin{corollary} The $L$-function of a cusp form of weight $k$ converges
  absolutely (and uniformly on compact subsets) for $\Re(s)>k/2+1$.
\end{corollary}

\begin{remark} Deligne's deep result mentioned above on the third Ramanujan
conjecture implies that we have the following optimal bound: there exists
$c>0$ such that $|a(n)|\le c\sigma_0(n)n^{(k-1)/2}$, and in particular
$|a(n)|=O(n^{(k-1)/2+\eps})$ for all $\eps>0$. This implies that the
$L$-function of a cusp form converges absolutely and uniformly on compact
subsets in fact also for $\Re(s)>(k+1)/2$.
\end{remark}

\begin{exercise}. Define for all $s\in\C$ the
function $\sigma_s(n)$ by $\sigma_s(n)=\sum_{d\mid n}d^s$ if $n\in\Z_{>0}$,
$\sigma_s(0)=\zeta(-s)/2$ (and $\sigma_s(n)=0$ otherwise). Set
$$S(s_1,s_2;n)=\sum_{0\le m\le n}\sigma_{s_1}(m)\sigma_{s_2}(n-m)\;.$$
\begin{enumerate}\item Compute $S(s_1,s_2;n)$ exactly in terms of
$\sigma_{s_1+s_2+1}(n)$ for $(s_1,s_2)=(3,3)$ and $(3,5)$, and also for
$(s_1,s_2)=(1,1)$, $(1,3)$, $(1,5)$, and $(1,7)$ by using properties of the
function $E_2$.
\item Using Hecke's bound for cusp forms, show that if $s_1$ and $s_2$ are
odd positive integers the ratio $S(s_1,s_2;n)/\sigma_{s_1+s_2+1}(n)$ tends
to a limit $L(s_1,s_2)$ as $n\to\infty$, and compute this limit in terms of
Bernoulli numbers. In addition, give an estimate for the \emph{error term}
$|S(s_1,s_2;n)/\sigma_{s_1+s_2+1}(n)-L(s_1,s_2)|$.
\item Using the values of the Riemann zeta function at even positive integers
in terms of Bernoulli numbers, show that if $s_1$ and $s_2$ are odd positive
integers we have
$$L(s_1,s_2)=\dfrac{\zeta(s_1+1)\zeta(s_2+1)}{(s_1+s_2+1)\binom{s_1+s_2}{s_1}\zeta(s_1+s_2+2)}\;.$$
\item (A little project.) \emph{Define} $L(s_1,s_2)$ by the above formula for
all $s_1$, $s_2$ in $\C$ for which it makes sense, interpreting
$\binom{s_1+s_2}{s_1}$ as $\Gamma(s_1+s_2+1)/(\Gamma(s_1+1)\Gamma(s_2+1))$.
Check on a computer whether it still seems to be true that
$$S(s_1,s_2;n)/\sigma_{s_1+s_2+1}(n)\to L(s_1,s_2)\;.$$
Try to \emph{prove} it for $s_1=s_2=2$, and then for general $s_1$, $s_2$.
If you succeed, give also an estimate for the error term analogous to the one
obtained above.
\end{enumerate}\end{exercise}

\smallskip

We now do some (elementary) analysis.

\begin{proposition}\label{propga} Let $F\in S_k(\G)$. For $\Re(s)>k/2+1$ we
  have
  $$(2\pi)^{-s}\Gamma(s)L(F,s)=\int_0^{\infty}F(it)t^{s-1}\,dt\;.$$
\end{proposition}

\begin{proof} Using $\Gamma(s)=\int_0^{\infty}e^{-t}t^{s-1}\,dt$, this is trivial
by uniform convergence which insures that we can integrate term by term.\fp\end{proof}

\begin{corollary}\label{corfeq} The function $L(F,s)$ is a holomorphic
function which can be analytically continued to the whole of $\C$. In addition,
if we set $\Lambda(F,s)=(2\pi)^{-s}\Gamma(s)L(F,s)$ we have the
functional equation $\Lambda(F,k-s)=i^{-k}\Lambda(F,s)$.
\end{corollary}

Note that in our case $k$ is even, so that $i^{-k}=(-1)^{k/2}$, but we prefer
writing the constant as above so as to be able to use a similar result in
odd weight, which occur in more general situations.

\begin{proof} Indeed, splitting the integral at $1$, changing $t$ into $1/t$ in one of
the integrals, and using modularity shows immediately that
$$(2\pi)^{-s}\Gamma(s)L(F,s)=\int_1^{\infty}F(it)(t^{s-1}+i^kt^{k-1-s})\,dt\;.$$
Since the integral converges absolutely and uniformly for all $s$
(recall that $F(it)$ tends exponentially fast to $0$ when $t\to\infty$), this
immediately implies the corollary.\fp\end{proof}

As an aside, note that the integral formula used in the
above proof is a very efficient numerical method to compute $L(F,s)$, since
the series obtained on the right by term by term integration is exponentially
convergent. For instance:

\begin{exercise} Let $F(\tau)=\sum_{n\ge1}a(n)q^n$ be the Fourier expansion of
a cusp form of weight $k$ on $\G$. Using the above formula, show that
the value of $L(F,k/2)$ at the center of the ``critical strip''
$0\le\Re(s)\le k$ is given by the following exponentially convergent series
$$L(F,k/2)=(1+(-1)^{k/2})\sum_{n\ge1}\dfrac{a(n)}{n^{k/2}}e^{-2\pi n}P_{k/2}(2\pi n)\;,$$
where $P_{k/2}(X)$ is the polynomial
$$P_{k/2}(X)=\sum_{0\le j<k/2}X^j/j!=1+X/1!+X^2/2!+\cdots+X^{k/2-1}/(k/2-1)!\;.$$
Note in particular that if $k\equiv2\pmod4$ we have $L(F,k/2)=0$. Prove this
directly.
\end{exercise}

\begin{exercise}\begin{enumerate}\item Prove that if $F$ is not necessarily a
cusp form we have $|a(n)|\le cn^{k-1}$ for some $c>0$.
\item Generalize the proposition and the integral formulas so that
they are also valid form non-cusp forms; you will have to add polar parts
of the type $1/s$ and $1/(s-k)$.
\item Show that $L(F,s)$ still extends to the whole of $\C$ with functional
equation, but that it has a pole, simple, at $s=k$, and compute its residue.
In passing, show that $L(F,0)=-a(0)$.
\end{enumerate}
\end{exercise}

\subsection{Special Values of $L$-Functions}

A general ``paradigm'' on $L$-functions, essentially due to P.~Deligne,
is that if some ``natural'' $L$-function has both an Euler product and
functional equations similar to the above, then for suitable integral
``special points'' the value of the $L$-function should be a certain
(a priori transcendental) number $\omega$ times an algebraic number.

In the case of modular forms, this is a theorem of Yu.~Manin:

\begin{theorem} Let $F$ be a normalized eigenform in $S_k(\G)$, and denote
  by $K$ the number field generated by its Fourier coefficients. There
  exist two nonzero complex numbers $\omega_{+}$ and $\omega_{-}$ such that
  for $1\le j\le k-1$ integral we have
  $$\Lambda(F,j)/\omega_{(-1)^j}\in K\;,$$
  where we recall that $\Lambda(F,s)=(2\pi)^{-s}\G(s)L(F,s)$.
  
  In addition, $\omega_{\pm}$ can be chosen such that
  $\omega_{+}\omega_{-}=<F,F>$.
\end{theorem}

In other words, for $j$ odd we have $L(F,j)/\om_{-}\in K$ while for
$j$ even we have $L(F,j)/\om_{+}\in K$.

For instance, in the case $F=\Delta$, if we choose
$\omega_{-}=\Lambda(F,3)$ and $\omega_{+}=\Lambda(F,2)$, we have
\begin{align*}
  (\Lambda(F,j))_{1\le j\le 11\ odd}&=(1620/691,1,9/14,9/14,1,1620/691)\omega_{-}\\
  (\Lambda(F,j))_{1\le j\le 11\ even}&=(1,25/48,5/12,25/48,1)\omega_{+}\;,
\end{align*}

and $\omega_{+}\omega_{-}=(8192/225)<F,F>$.

\begin{exercise} (see also Exercise \ref{ex:Bol}). For $F\in S_k(\G)$ define
  the \emph{period polynomial} $P(F,X)$ by
  $$P(F;X)=\int_0^{i\infty}(X-\tau)^{k-2}F(\tau)\,d\tau\;.$$
  \begin{enumerate}
  \item For $\ga\in\G$ show that
    $$P(F;X)|_{2-k}=\int_{\ga^{-1}(0)}^{\ga^{-1}(i\infty)}(X-\tau)^{k-2}F(\tau)\,d\tau\;.$$
  \item Show that $P(F;X)$ satisfies
    \begin{align*}
      &P(F;X)|_{2-k}S+P(F;X)=0\text{\quad and}\\
      &P(F;X)|_{2-k}(ST)^2+P(F;X)|_{2-k}(ST)+P(F;X)=0\;.
    \end{align*}
  \item Show that
    $$P(F;X)=-\sum_{j=0}^{k-2}(-i)^{k-1-j}\binom{k-2}{j}\Lambda(F,k-1-j)X^j\;.$$
  \item If $F=\Delta$, using Manin's theorem above show that up to
    the multiplicative constant $\om_+$, $\Re(P(F;X))$
    factors completely in $\Q[X]$ as a product of linear polynomials, and
    show a similar result for $\Im(P(F;X))$ after omitting the extreme terms
    involving $691$.
\end{enumerate}\end{exercise}

\subsection{Nonanalytic Eisenstein Series and Rankin--Selberg}

If we replace the expression $(c\tau+d)^k$ by $|c\tau+d|^{2s}$ for some
complex number $s$, we can also obtain functions which are invariant by
$\G$, although they are nonanalytic. More precisely:

\begin{definition}\label{def:nonhol} Write as usual $y=\Im(\tau)$. For
$\Re(s)>1$ we define
  \begin{align*}G(s)(\tau)&=\sum_{(c,d)\in\Z^2\setminus\{(0,0)\}}\dfrac{y^s}{|c\tau+d|^{2s}}\text{\quad and}\\
    E(s)(\tau)&=\sum_{\ga\in\G_\infty\backslash\G}\Im(\ga(\tau))^s=\dfrac{1}{2}\sum_{\gcd(c,d)=1}\dfrac{y^s}{|c\tau+d|^{2s}}\;.\end{align*}\end{definition}

This is again an \emph{averaging} procedure, and it follows that
$G(s)$ and $E(s)$ are \emph{invariant} under $\G$. In addition, as in the case
of the holomorphic Eisenstein series $G_k$ and $E_k$, it is clear that
$G(s)=\zeta(2s)E(s)$. One can also easily compute their Fourier expansion,
and the result is as follows:

\begin{proposition} Set $\Lambda(s)=\pi^{-s/2}\G(s/2)\z(s)$. We have the
  Fourier expansion
  $$\Lambda(2s)E(s)=\Lambda(2s)y^s+\Lambda(2-2s)y^{1-s}+4y^{1/2}\sum_{n\ge1}\dfrac{\sigma_{2s-1}(n)}{n^{s-1/2}}K_{s-1/2}(2\pi ny)\cos(2\pi nx)\;.$$
\end{proposition}

In the above, $K_{\nu}(x)$ is a $K$-Bessel function which we do not define
here. The main properties that we need is that it tends to $0$ exponentially
(more precisely $K_{\nu}(x)\sim(\pi/(2x))^{1/2}e^{-x}$ as $x\to\infty$) and
that $K_{-\nu}=K_{\nu}$. It follows from the above Fourier expansion that
$E(s)$ has an \emph{analytic continuation} to the whole complex plane,
that it satisfies the functional equation $\E(1-s)=\E(s)$, where we set
$\E(s)=\Lambda(2s)E(s)$, and that $E(s)$ has a unique pole, at $s=1$, which is
simple with residue $3/\pi$, independent of $\tau$.

\begin{exercise} Using the properties of the Riemann zeta function $\z(s)$,
  show this last property, i.e., that $E(s)$ has a unique pole, at $s=1$,
  which is simple with residue $3/\pi$, independent of $\tau$.\end{exercise}

There are many reasons for introducing these nonholomorphic Eisenstein series,
but for us the main reason is that they are fundamental in \emph{unfolding}
methods. Recall that using unfolding, in Proposition \ref{prop:unfold}
we showed that $E_k$ (or $G_k$) was orthogonal to any cusp form. In the
present case, we obtain a different kind of result called a
\emph{Rankin--Selberg convolution}. Let $f$ and $g$ be in $M_k(\G)$, one
of them being a cusp form. Since $E(s)$ is invariant by $\G$ the scalar
product $<E(s)f,g>$ makes sense, and the following proposition
gives its value:

\begin{proposition} Let $f(\tau)=\sum_{n\ge0}a(n)q^n$ and
  $g(\tau)=\sum_{n\ge0}b(n)q^n$ be in $M_k(\G)$, with at least one being a
  cusp form. For $\Re(s)>1$ we have
  $$<E(s)f,g>=\dfrac{\G(s+k-1)}{(4\pi)^{s+k-1}}\sum_{n\ge1}\dfrac{a(n)\ov{b(n)}}{n^{s+k-1}}\;.$$
\end{proposition}

\begin{proof} We essentially copy the proof of Proposition \ref{prop:unfold} so
we skip the details: setting temporarily $F(\tau)=f(\tau)\ov{g(\tau)}y^k$
which is invariant by $\G$, we have
\begin{align*}<E(s)f,g>&=\int_{\G\backslash\H}\sum_{\ga\in\G_\infty\backslash\G}\Im(\ga(\tau))^sF(\ga(\tau))\,d\mu\\
  &=\sum_{\G_\infty\backslash\H}\Im(\tau)^sF(\tau)\,d\mu\\
  &=\int_0^\infty y^{s+k-2}\int_0^1F(x+iy)\,dx\,dy\;.\end{align*}

The inner integral is equal to the constant term in the Fourier expansion of
$F$, hence is equal to $\sum_{n\ge1}a(n)\ov{b(n)}e^{-4\pi ny}$ (note that by
assumption one of $f$ and $g$ is a cusp form, so the term $n=0$ vanishes), and
the proposition follows.\fp\end{proof}

\begin{corollary} For $\Re(s)>k$ set
  $$R(f,g)(s)=\sum_{n\ge1}\dfrac{a(n)\ov{b(n)}}{n^s}\;.$$
  \begin{enumerate}
  \item $R(f,g)(s)$ has an analytic continuation to the whole complex plane
    and satisfies the functional equation $\RC(2k-1-s)=\RC(s)$ with
    $$\RC(s)=\Lambda(2s-2k+1)(4\pi)^{-s}\G(s)R(f,g)(s)\;.$$
  \item $R(f,g)(s)$ has a single pole, which is simple, at $s=k$ with
    residue
    $$\dfrac{3}{\pi}\dfrac{(4\pi)^k}{(k-1)!}<f,g>\;.$$
  \end{enumerate}
\end{corollary}

\begin{proof} This immediately follows from the corresponding properties of
  $E(s)$: we have
$$\Lambda(2s-2k+2)(4\pi)^{-s}\G(s)R(f,g)(s)=<\E(s-k+1)f,g>\;,$$
and the right-hand side has an analytic continuation to $\C$, is invariant
when changing $s$ into $2k-1-s$. In addition by the proposition
$E(s-k+1)=\E(s-k+1)/\Lambda(2s-2k+2)$ has a single pole, which is simple,
at $s=k$, with residue $3/\pi$, so $R(f,g)(s)$ also has a single pole, which
is simple, at $s=k$ with residue
$\dfrac{3}{\pi}\dfrac{(4\pi)^k}{(k-1)!}<f,g>$.\fp\end{proof}

It is an important fact (see Theorem 7.9 of my notes on $L$-functions in the
present volume) that $L$-functions having analytic continuation and
standard functional equations can be very efficiently computed at any point
in the complex plane (see the note after the proof of Corollary
\ref{corfeq} for the special case of $L(F,s)$). Thus the above corollary gives
a very efficient method for computing Petersson scalar products.

Note that the \emph{holomorphic} Eisenstein series $E_k(\tau)$ can also be used
to give Rankin--Selberg convolutions, but now between forms of different
weights:

\begin{exercise} Let $f=\sum_{n\ge0}a(n)q^n\in M_{\ell}(\G)$ and
$g=\sum_{n\ge0}b(n)q^n\in M_{k+\ell}(\G)$, at least one being a cusp form.
Using exactly the same unfolding method as in the above proposition or as in
Proposition \ref{prop:unfold}, show that
$$<E_kf,g>=\dfrac{(k+\ell-2)!}{(4\pi)^{k+\ell-1}}\sum_{n\ge1}\dfrac{a(n)\ov{b(n)}}{n^{k+\ell-1}}\;.$$\end{exercise}

\section{Modular Forms on Subgroups of $\G$}

\subsection{Types of Subgroups}

We have used as basic definition of (weak) modularity $F|_k\ga=F$ for all
$\ga\in\G$. But there is no reason to restrict to $\G$: we could very well
ask the same modularity condition for some group $G$ of transformations of
$\H$ different from $\G$.

There are many types of such groups, and they have been classified: for us,
we will simply distinguish three types, with no justification. For any such
group $G$ we can talk about a fundamental domain, similar to $\FF$ that
we have drawn above (I do not want to give a rigorous definition here).
We can distinguish essentially three types of such domains, corresponding
to three types of groups.

The first type is when the domain (more precisely
its closure) is \emph{compact}: we say in that case that $G$ is
\emph{cocompact}. It is equivalent to saying that it does not have any
``cusp'' such as $i\infty$ in the case of $G$. These groups are very important,
but we will not consider them here.

The second type is when the domain is not compact (i.e., it has cusps), but it
has \emph{finite volume} for the measure $d\mu=dxdy/y^2$ on $\H$ defined
in Exercise \ref{ex:dmu}. Such a group is said to have finite \emph{covolume},
and the main example is $G=\G$ that we have just considered, hence also
evidently all the subgroups of $\G$ of \emph{finite index}.

\begin{exercise} Show that the covolume of the modular group $\G$ is finite
and equal to $\pi/3$.\end{exercise}

The third type is when the volume is infinite: a typical example is the
group $\G_{\infty}$ generated by integer translations, i.e., the set of
matrices $\psmm{1}{n}{0}{1}$. A fundamental domain is then any vertical strip
in $\H$ of width $1$, which can trivially be shown to have infinite volume.
These groups are not important (at least for us) for the following reason:
they would have ``too many'' modular forms. For instance, in the case of
$\G_{\infty}$ a ``modular form'' would simply be a holomorphic periodic
function of period $1$, and we come back to the theory of Fourier series,
much less interesting.

We will therefore restrict to groups of the second type, which are called
\emph{Fuchsian groups of the first kind}. In fact, for this course we will
even restrict to subgroups $G$ of $\G$ of \emph{finite index}.

\medskip

However, even with this restriction, it is still necessary to distinguish
two types of subgroups: the so-called \emph{congruence subgroups}, and the
others, of course called non-congruence subgroups. The theory of modular
forms on non-congruence subgroups is quite a difficult subject and active
research is being done on them. One annoying aspect is that they apparently
do not have a theory of Hecke operators.

Thus will will restrict even more to congruence subgroups. We give the
following definitions:

\begin{definition} Let $N\ge1$ be an integer.
  \begin{enumerate}
  \item We define
    \begin{align*}
      \G(N)&=\{\ga=\begin{pmatrix}a & b\\ c & d\end{pmatrix}\in\G,\ \ga\equiv
      \begin{pmatrix}1 & 0\\ 0 & 1\end{pmatrix}\pmod{N}\}\;,\\
      \G_1(N)&=\{\ga=\begin{pmatrix}a & b\\ c & d\end{pmatrix}\in\G,\ \ga\equiv
      \begin{pmatrix}1 & *\\ 0 & 1\end{pmatrix}\pmod{N}\}\;,\\
      \G_0(N)&=\{\ga=\begin{pmatrix}a & b\\ c & d\end{pmatrix}\in\G,\ \ga\equiv
      \begin{pmatrix}* & *\\ 0 & *\end{pmatrix}\pmod{N}\}\;,\\
    \end{align*}
    where the congruences are component-wise and $*$ indicates that no congruence
    is imposed.
  \item A subgroup of $\G$ is said to be a \emph{congruence subgroup} if
    it contains $\G(N)$ for some $N$, and the smallest such $N$ is called
    the \emph{level} of the subgroup.
  \end{enumerate}
\end{definition}

It is clear that $\G(N)\subset\G_1(N)\subset\G_0(N)$, and it is trivial to
prove that $\G(N)$ is normal in $\G$ (hence in any subgroup of $\G$ containing
it), that $\G_1(N)/\G(N)\isom\Z/N\Z$ (with the map
$\psmm{a}{b}{c}{d}\mapsto b\bmod N$), and that $\G_1(N)$ is normal in
$\G_0(N)$ with $\G_0(N)/\G_1(N)\isom(\Z/N\Z)^*$ (with the map
$\psmm{a}{b}{c}{d}\mapsto d\bmod N$).

If $G$ is a congruence subgroup of level $N$ we have $\G(N)\subset G$,
so (whatever the definition) a modular form on $G$ will in particular be
on $\G(N)$. Because of the above isomorphisms, it is not difficult to reduce
the study of forms on $\G(N)$ to those on $\G_1(N)$, and the latter to
forms on $\G_0(N)$, except that we have to add a slight ``twist'' to the
modularity property. Thus for simplicity, we will restrict to modular forms on
$\G_0(N)$.

\subsection{Modular Forms on Subgroups}

In view of the definition given for $\G$, it is natural to say that
$F$ is weakly modular of weight $k$ on $\G_0(N)$ if for all $\ga\in\G_0(N)$
we have $F|_k\ga=F$, where we recall that if $\ga=\psmm{a}{b}{c}{d}$ then
$F|_k\ga(\tau)=(c\tau+d)^{-k}F(\tau)$. To obtain a modular \emph{form}, we
need also to require that $F$ is holomorphic on $\H$, plus some additional
technical condition ``at infinity''. In the case of the full modular group
$\G$, this condition was that $F(\tau)$ remains bounded as
$\Im(\tau)\to\infty$. In the case of a subgroup, this condition is not
sufficient (it is easy to show that if we do not require an additional
condition the corresponding space will in general be infinite-dimensional).
There are
several equivalent ways of giving the additional condition. One is the
following: writing as usual $\tau=x+iy$, we require that there exists $N$
such that in the strip $-1/2\le x\le 1/2$, we have
$|F(\tau)|\le y^N$ as $y\to\infty$
and $|F(\tau)|\le y^{-N}$ as $y\to0$ (since $F$ is $1$-periodic, there is no
loss of generality in restricting to the strip).

It is easily shown that if $F$ is weakly modular and holomorphic, then
the above inequalities imply that $|F(\tau)|$ is in fact \emph{bounded}
as $y\to\infty$ (but in general \emph{not} as $y\to0$), so the first condition
is exactly the one that we gave in the case of the full modular group.

Similarly, we can define a \emph{cusp form} by asking that in the above
strip $|F(\tau)|$ tends to $0$ as $y\to\infty$ and as $y\to0$.

\begin{exercise} If $F\in M_k(\G)$ show that the second condition
$|F(\tau)|\le y^{-N}$ as $y\to0$ is satisfied.
\end{exercise}

Now that we have a solid definition of modular form, we can try to proceed
as in the case of the full modular group. A number of things can easily be
generalized. It is always convenient to choose a system of representatives
$(\ga_j)$ of right cosets for $\G_0(N)$ in $\G$, so that
$$\G=\bigsqcup_j\G_0(N)\ga_j\;.$$
For instance, if $\FF$ is the fundamental domain of $\G$ seen above, one
can choose ${\cal D}=\bigsqcup\ga_j(\FF)$ as fundamental domain for $\G_0(N)$.
The theorem that we gave on valuations generalizes immediately:
$$\sum_{\tau\in\ov{{\cal D}}}\dfrac{v_{\tau}(F)}{e_{\tau}}=\dfrac{k}{12}[\G:\G_0(N)]\;,$$
where $\ov{{\cal D}}$ is ${\cal D}$ to which is added a finite number of
``cusps'' (we do not explain this; it is \emph{not} the topological closure),
$e_{\tau}=2$ (resp., $3$) if $\tau$ is $\G$-equivalent to $i$ (resp., to
$\rho$), and $e_{\tau}=1$ otherwise, and we can then deduce the dimension of
$M_k(\G_0(N))$ and $S_k(\G_0(N))$ as we did for $\G$:
    
\begin{theorem}\label{thmdim} We have $M_0(\G_0(N))=\C$ (i.e., the only modular
  forms of weight $0$ are the constants) and $S_0(\G_0(N))=\{0\}$. For $k\ge2$
  even, we have
  \begin{align*}\dim(M_k(\G_0(N)))&=A_1-A_{2,3}-A_{2,4}+A_3\text{\quad and}\\
    \dim(S_k(\G_0(N)))&=A_1-A_{2,3}-A_{2,4}-A_3+\delta_{k,2}\;,\end{align*}
  where $\delta_{k,2}$ is the Kronecker symbol ($1$ if $k=2$, $0$ otherwise)
  and the $A_i$ are given as follows:
  \begin{align*}
    A_1&=\dfrac{k-1}{12}N\prod_{p\mid N}\left(1+\dfrac{1}{p}\right)\;,\\
    A_{2,3}&=\left(\dfrac{k-1}{3}-\left\lfloor\dfrac{k}{3}\right\rfloor\right)\prod_{p\mid N}\left(1+\leg{-3}{p}\right)\text{\quad if\ \ $9\nmid N$,\quad $0$\ \ otherwise,}\\
    A_{2,4}&=\left(\dfrac{k-1}{4}-\left\lfloor\dfrac{k}{4}\right\rfloor\right)\prod_{p\mid N}\left(1+\leg{-4}{p}\right)\text{\quad if\ \ $4\nmid N$,\quad $0$\ \ otherwise,}\\
    \bigskip
    A_3&=\dfrac{1}{2}\sum_{d\mid N}\phi(\gcd(d,N/d))\;.
  \end{align*}
\end{theorem}

\subsection{Examples of Modular Forms on Subgroups}

We give a few examples of modular forms on subgroups. First note the following
easy lemma:

\begin{lemma}\label{lembd} If $F\in M_k(\G_0(N))$ then for any $m\in\Z_{\ge1}$
  we have $F(m\tau)\in M_k(\G_0(mN))$.\end{lemma}

\begin{proof} Trivial since when $\psmm{a}{b}{c}{d}\in\G_0(mN)$ one can write
$(m(a\tau+b)/(c\tau+d))=(a(m\tau)+mb)/((c/m)\tau+d)$.\fp\end{proof}

Thus we can already construct many forms on subgroups, but in a sense they
are not very interesting, since they are ``old'' in a precise sense that we
will define below.

\smallskip

A second more interesting example is Eisenstein series: there are more
general Eisenstein series than those that we have seen for $\G$, but we simply
give the following important example: using a similar proof to the above
lemma we can construct Eisenstein series of \emph{weight $2$} as follows.
Recall that $E_2(\tau)=1-24\sum_{n\ge1}\sigma_1(n)q^n$ is not quite modular,
and that $E_2^*(\tau)=E_2(\tau)-3/(\pi\Im(\tau))$ is weakly modular (but of
course non-holomorphic). Consider the function $F(\tau)=NE_2(N\tau)-E_2(\tau)$,
analogous to the construction of the lemma with a correction term.

We have the evident but crucial fact that we also have
$F(\tau)=NE_2^*(N\tau)-E_2^*(\tau)$ (since $\Im(\tau)$ is multiplied by $N$),
so $F$ is also weakly modular on $\G_0(N)$, but since it is holomorphic we
have thus constructed a (nonzero) modular form of weight $2$ on $\G_0(N)$.

\smallskip

A third important example is provided by theta series. This would require
a book in itself, so we restrict to the simplest case. We have seen in
Corollary \ref{corth} that the function $T(a)=\sum_{n\in\Z}e^{-a\pi n^2}$
satisfies $T(1/a)=a^{1/2}T(a)$, which looks like (and is) a modularity
condition. This was for $a>0$ real. Let us generalize and for $\tau\in\H$ set
$$\th(\tau)=\sum_{n\in\Z}q^{n^2}=\sum_{n\in\Z}e^{2\pi in^2\tau}\;,$$
so that for instance we simply have $T(a)=\th(ia/2)$. The proof of the
functional equation for $T$ that we gave using Poisson summation is still
valid in this more general case and shows that
$$\th(-1/(4\tau))=(2\tau/i)^{1/2}\th(\tau)\;.$$
On the other hand, the definition trivially shows that $\th(\tau+1)=\th(\tau)$.
If we denote by $W_4$ the matrix $\psmm{0}{-1}{4}{0}$ corresponding
to the map $\tau\mapsto-1/(4\tau)$ and as usual $T=\psmm{1}{1}{0}{1}$, we
thus have $\th|_{1/2}W_4=c\th$ and $\th_{1/2}T=\th$ for some $8$th root of
unity $c$. (Note: we always use the principal determination of the square
roots; if you are uncomfortable with this, simply square everything, this is
what we will do below anyway.) This implies that if we let $\G_{\th}$ be the
intersection of $\G$ with the group generated by $W_4$ and $T$ (as
transformations of $\H$), then for all $\ga\in\G_{\th}$ we will have
$\th|_{1/2}\ga=c(\ga)\th$ for some $8$th root of unity $c(\ga)$, but in fact
$c(\ga)$ is a $4$th root of unity which we will give explicitly below.

One can easily describe this group $\G_{\th}$, and in particular show that
it contains $\G_0(4)$ as a subgroup of index $2$. This implies that
$\th^4\in M_2(\G_0(4))$, and more generally of course
$\th^{4m}\in M_{2m}(\G_0(4))$.

As one of the most famous application of the finite-dimensionality of
modular form spaces, solve the following exercise:

\begin{exercise}\label{exr4r8}\begin{enumerate}
\item Using the dimension formulas, show that $2E_2(2\tau)-E_2(\tau)$
  together with $4E_2(4\tau)-E_2(\tau)$ form a basis of $M_2(\G_0(4))$.
\item Using the Fourier expansion of $E_2$, deduce an explicit formula for
  the Fourier expansion of $\th^4$, and hence that $r_4(n)$, the number of
  representations of $n$ as a sum of $4$ squares (in $\Z$, all permutations
  counted) is given for $n\ge1$ by the formula
  $$r_4(n)=8(\sigma_1(n)-4\sigma_1(n/4))\;,$$
  where it is understood that $\sigma_1(x)=0$ if $x\notin\Z$. In
  particular, show that this trivially implies Lagrange's theorem that every
  integer is a sum of four squares.
\item Similarly, show that $r_8(n)$, the $n$th Fourier coefficient of
  $\th^8$, is given for $n\ge1$ by
  $$r_8(n)=16(\sigma_3(n)-2\sigma_3(n/2)+16\sigma_3(n/4))\;.$$
\end{enumerate}
\end{exercise}

\begin{remark} Using more general methods one can give ``closed'' formulas for
$r_k(n)$ for $k=1$, $2$, $3$, $4$, $5$, $6$, $7$, $8$, and $10$, see e.g.,
\cite{Coh-Str}.\end{remark}

\subsection{Hecke Operators and $L$-Functions}

We can introduce the same Hecke operators as before, but to have a reasonable
definition we must add a coprimality condition: we define
$T(n)(\sum_{m\ge0}a(m)q^m)=\sum_{m\ge0}b(m)q^m$, with
$$b(m)=\sum_{\substack{d\mid\gcd(m,n)\\\gcd(d,N)=1}}d^{k-1}a(mn/d^2)\;.$$
This additional condition $\gcd(d,N)=1$ is of course automatically
satisfied if $n$ is coprime to $N$, but not otherwise.

One then shows exactly like in the case of the full modular group that
$$T(n)T(m)=\sum_{\substack{d\mid\gcd(n,m)\\\gcd(d,N)=1}}d^{k-1}T(nm/d^2)\;,$$
that they preserve modularity, so in particular the $T(n)$ form a commutative
algebra of operators on $S_k(\G_0(N))$. And this is where the difficulties
specific to subgroups of $\G$ begin: in the case of $\G$ we stated (without
proof nor definition) that the $T(n)$ were \emph{Hermitian} with respect to
the Petersson scalar product, and deduced the existence of eigenforms for
\emph{all} Hecke operators. Unfortunately here the same proof shows that
the $T(n)$ are Hermitian when $n$ is coprime to $N$, but \emph{not} otherwise.

It follows that there exist common eigenforms for the $T(n)$, but \emph{only}
for $n$ coprime to $N$, which creates difficulties.

An analogous problem occurs for \emph{Dirichlet characters}: if $\chi$
is a Dirichlet character modulo $N$, it may in fact come by natural extension
from a character modulo $M$ for some divisor $M\mid N$, $M<N$. The
characters which have nice properties, in particular with respect to the
functional equation of their $L$-functions, are the \emph{primitive}
characters, for which such an $M$ does not exist.

A similar but slightly more complicated thing can be done for modular forms.
It is clear that if $M\mid N$ and $F\in M_k(\G_0(M))$, then of course
$F\in M_k(\G_0(N))$. More generally, by Lemma \ref{lembd}, for any $d\mid N/M$
we have $F(d\tau)\in M_k(\G_0(N))$. Thus we want to exclude such ``oldforms''.
However it is not sufficient to say that a newform is not an oldform. The
correct definition is to define a newform as a form which is
\emph{orthogonal} to the space of oldforms with respect to the
scalar product, and of course the new space is the
space of newforms. Note that in the case of Dirichlet characters this
orthogonality condition (for the standard scalar product of two characters)
is automatically satisfied so need not be added.

This theory was developed by Atkin--Lehner--Li, and the new space
$S_k^{\text{new}}(\G_0(N))$ can be shown to have all the nice properties
that we require. Although not trivial, one can prove that it has a basis
of common eigenforms for \emph{all} Hecke operators, not only those with
$n$ coprime to $N$. More precisely, one shows that in the new space an
eigenform for the $T(n)$ for all $n$ coprime to $N$ is automatically
an eigenform for \emph{any} operator which commutes with all the $T(n)$,
such as, of course, the $T(m)$ for $\gcd(m,N)>1$.

In addition, we have not really lost anything by
restricting to the new space, since it is easy to show that
$$S_k(\G_0(N))=\bigoplus_{M\mid N}\bigoplus_{d\mid N/M}B(d)S_k^{\text{new}}(\G_0(M))\;,$$
where $B(d)$ is the operator sending $F(\tau)$ to $F(d\tau)$. Note that the
sums in the above formula are \emph{direct} sums.

\begin{exercise} The above formula shows that
$$\dim(S_k(\G_0(N)))=\sum_{M\mid N}\sigma_0(N/M)\dim(S_k^{\text{new}}(\G_0(M)))\;,$$
where $\sigma_0(n)$ is the number of divisors of $n$.
\begin{enumerate}\item Using the M\"obius inversion formula, show that if we
define an arithmetic function $\be$ by $\be(p)=-2$, $\be(p^2)=1$, and
$\be(p^k)=0$ for $k\ge3$, and extend by multiplicativity
($\be(\prod p_i^{v_i})=\prod\be(p_i^{v_i})$), we have the following dimension
formula for the new space:
$$\dim(S_k^{\text{new}}(\G_0(N)))=\sum_{M\mid N}\be(N/M)\dim(S_k(\G_0(M)))\;.$$
\item Using Theorem \ref{thmdim}, deduce a direct formula for the dimension
  of the new space.
\end{enumerate}
\end{exercise}

\begin{proposition} Let $F\in S_k(\G_0(N))$ and $W_N=\psmm{0}{-1}{N}{0}$.
  \begin{enumerate}\item We have $F|_kW_N\in S_k(\G_0(N))$, where
    $$F|_kW_N(\tau)=N^{-k/2}\tau^{-k}F(-1/(N\tau))\;.$$
  \item If $F$ is an eigenform (in the new space) then $F|_kW_N=\pm F$ for a
    suitable sign $\pm$.
    \end{enumerate}
  \end{proposition}

\begin{proof} (1): this simply follows from the fact that $W_N$ \emph{normalizes}
$\G_0(N)$: $W_N^{-1}\G_0(N)W_N=\G_0(N)$ as can easily be checked, and the
same result would be true for any other normalizing operator such as the
\emph{Atkin--Lehner} operators which we will not define. The operator $W_N$
is called the \emph{Fricke involution}.

(2): It is easy to show that $W_N$ commutes with all Hecke operators $T(n)$
when $\gcd(n,N)=1$, so by what we have mentioned above, if $F$ is an
eigenform in the new space it is automatically an eigenform for $W_N$,
and since $W_N$ acts as an involution, its eigenvalues are $\pm1$.\fp\end{proof}

The eigenforms can again be normalized with $a(1)=1$, and their $L$-function
has an Euler product, of a slightly more general shape:
$$L(F,s)=\prod_{p\nmid N}\dfrac{1}{1-a(p)p^{-s}+p^{k-1}p^{-2s}}\prod_{p\mid N}\dfrac{1}{1-a(p)p^{-s}}\;.$$
Proposition \ref{propga} is of course still valid, but is not the correct
normalization to obtain a functional equation. We replace it by
$$N^{s/2}(2\pi)^{-s}\Gamma(s)L(F,s)=\int_0^{\infty}F(it/N^{1/2})t^{s-1}\,dt\;,$$
which of course is trivial from the proposition by replacing $t$ by
$t/N^{1/2}$. Indeed, thanks to the above proposition we split the integral
at $t=1$, and using the action of $W_N$ we deduce the following proposition:

\begin{proposition} Let $F\in S_k^{\text{new}}(\G_0(N))$ be an eigenform for
  all Hecke operators, and write $F|_kW_N=\eps F$ for some $\eps=\pm1$.
  The $L$-function $L(F,s)$ extends to a holomorphic function in $\C$, and
  if we set $\Lambda(F,s)=N^{s/2}(2\pi)^{-s}\Gamma(s)L(F,s)$ we have the
  functional equation
  $$\Lambda(F,k-s)=\eps i^{-k}\Lambda(F,s)\;.$$
\end{proposition}

\begin{proof} Indeed, the trivial change of variable $t$ into $1/t$ proves the formula
$$N^{s/2}(2\pi)^{-s}\Gamma(s)L(F,s)=\int_1^{\infty}F(it/N^{1/2})(t^{s-1}+\eps i^kt^{k-1-s})\,dt\;,$$
from which the result follows.\fp\end{proof}

Once again, we leave to the reader to check that if
$F(\tau)=\sum_{n\ge1}a(n)q^n$ we have
$$L(F,k/2)=(1+\eps(-1)^{k/2})\sum_{n\ge1}\dfrac{a(n)}{n^{k/2}}e^{-2\pi n/N^{1/2}}P_{k/2}(2\pi n/N^{1/2})\;.$$

\subsection{Modular Forms with Characters}

Consider again the problem of sums of squares, in other words of the powers
of $\th(\tau)$. We needed to raise it to a power which is a multiple of $4$
so as to have a pure modularity property as we defined it above.
But consider the function $\th^2(\tau)$. The same proof that we mentioned for
$\th^4$ shows that for any $\ga=\psmm{a}{b}{c}{d}\in\G_0(4)$ we have
$$\th^2(\ga(\tau))=\leg{-4}{d}(c\tau+d)\th^2(\tau)\;,$$
where $\lgs{-4}{d}$ is the Legendre--Kronecker character (in this specific
case equal to $(-1)^{(d-1)/2}$ since $d$ is odd, being coprime to $c$).
Thus it satisfies a modularity property, except that it is ``twisted''
by $\lgs{-4}{d}$. Note that the equation makes sense since if we change
$\ga$ into $-\ga$ (which does not change $\ga(\tau)$), then $(c\tau+d)$
is changed into $-(c\tau+d)$, and $\lgs{-4}{d}$ is changed into
$\lgs{-4}{-d}=-\lgs{-4}{d}$. It is thus essential that the multiplier
that we put in front of $(c\tau+d)^k$, here $\lgs{-4}{d}$, has the same
parity as $k$.

We mentioned above that the study of modular forms on $\G_1(N)$ could be
reduced to those on $\G_0(N)$ ``with a twist''. Indeed, more precisely it is
trivial to show that
$$M_k(\G_1(N))=\bigoplus_{\chi(-1)=(-1)^k}M_k(\G_0(N),\chi)\;,$$
where $\chi$ ranges through all Dirichlet characters modulo $N$ of the
specified parity, and where $M_k(\G_0(N),\chi)$ is defined as the space of
functions $F$ satisfying
$$F(\ga(\tau))=\chi(d)(c\tau+d)^kF(\tau)$$ for all
$\ga=\psmm{a}{b}{c}{d}\in\G_0(N)$, plus the usual holomorphy and conditions
at the cusps (note that $\ga\mapsto\chi(d)$ is the group homomorphism
from $\G_0(N)$ to $\C^*$ which induces the above-mentioned isomorphism
from $\G_0(N)/\G_1(N)$ to $(\Z/N\Z)^*$).

\begin{exercise}\begin{enumerate}
  \item Show that a system of coset representatives of
    $\G_1(N)\backslash\G_0(N)$ is given by matrices $M_d=\psmm{u}{-v}{N}{d}$,
    where $0\le d<N$ such that $\gcd(d,N)=1$ and $u$ and $v$ are such that
    $ud+vN=1$.
  \item Let $f\in M_k(\G_1(N))$. Show that in the above decomposition
    of $M_k(\G_1(N))$ we have $f=\sum_{\chi(-1)=(-1)^k}f_{\chi}$ with
    $$f_{\chi}=\sum_{0\le d<N,\ \gcd(d,N)=1}\overline{\chi(d)}f|_kM_d\;.$$
\end{enumerate}\end{exercise}

These spaces are just as nice as the spaces $M_k(\G_0(N))$ and share exactly
the same properties. They have finite dimension (which we do not give),
there are Eisenstein series, Hecke operators, newforms, Euler products,
$L$-functions, etc... An excellent rule of thumb is simply to replace any
formula containing
$d^{k-1}$ (or $p^{k-1}$) by $\chi(d)d^{k-1}$ (or $\chi(p)p^{k-1}$). In fact,
in the Euler product of the $L$-function of an eigenform we do not need
to distinguish $p\nmid N$ and $p\mid N$ since we have
$$L(F,s)=\prod_{p\in P}\dfrac{1}{1-a(p)p^{-s}+\chi(p)p^{k-1-2s}}\;,$$
and $\chi(p)=0$ if $p\mid N$ since $\chi$ is a character modulo $N$.

Thus, for instance $\th^2\in M_1(\G_0(4),\chi_{-4})$, more generally
$\th^{4m+2}\in M_{2m+1}(\G_0(4),\chi_{-4})$, where we use the notation
$\chi_D$ for the Legendre--Kronecker symbol $\lgs{D}{d}$.

The space $M_1(\G_0(4),\chi_{-4})$ has dimension $1$, generated by the single
Eisenstein series
$$1+4\sum_{n\ge1}\sigma_0^{(-4)}(n)q^n\;,\text{\quad where\quad}\sigma_{k-1}^{(D)}(n)=\sum_{d\mid n}\leg{D}{d}d^{k-1}$$
according to our rule of thumb (which does not tell us the constant $4$).
Comparing constant coefficients, we deduce that $r_2(n)=4\sigma_0^{(-4)}(n)$,
where as usual $r_2(n)$ is the number of representations of $n$ as a sum of
two squares. This formula was in essence discovered by Fermat.

For $r_6(n)$ we must work slightly more: $\th^6\in M_3(\G_0(4),\chi_{-4})$,
and this space has dimension $2$, generated by two Eisenstein series.
The first is the natural ``rule of thumb'' one (which again does not give us
the constant)
$$F_1=1-4\sum_{n\ge1}\sigma_2^{(-4)}(n)q^n\;,$$
and the second is $$F_2=\sum_{n\ge1}\sigma_2^{(-4,*)}(n)q^n\;,$$
where $$\sigma_{k-1}^{(D,*)}=\sum_{d\mid n}\leg{D}{n/d}d^{k-1}\;,$$
a sort of dual to $\sigma_{k-1}^{(D)}$ (these are my notation).
Since $\th^6=1+12q+\cdots$, comparing the Fourier coefficients of $1$ and $q$
shows that $\th^6=F_1+16F_2$, so we deduce that
$$r_6(n)=-4\sigma_2^{(-4)}(n)+16\sigma_2^{(-4,*)}(n)
=\sum_{d\mid n}\left(16\leg{-4}{n/d}-4\leg{-4}{d}\right)d^2\;.$$

\subsection{Remarks on Dimension Formulas and Galois Representations}

The explicit dimension formulas alluded to above are valid for $k\in\Z$
\emph{except} for $k=1$; in addition, thanks to the theorems mentioned below,
we also have explicit dimension formulas for $k\in 1/2+\Z$. Thus, the
theory of modular forms of weight $1$ is very special, and their general
construction more difficult.

This is also reflected in the construction of \emph{Galois representations}
attached to modular eigenforms, which is an important and deep subject that we
will not mention in this course, except to say the following: in weight
$k\ge2$ these representations are $\ell$-adic (or modulo $\ell$), i.e., with
values in $\GL_2(\Q_\ell)$ (or $\GL_2(\F_\ell)$), while in weight $1$ they are
\emph{complex} representations, i.e., with values in $\GL_2(\C)$. The
construction in weight $2$ is quite old, and comes directly from the
construction of the so-called \emph{Tate module} $T(\ell)$
attached to an Abelian variety (more precisely the Jacobian of a modular
curve), while the construction in higher weight, due to Deligne, is much
deeper since it implies the third Ramanujan conjecture $|\tau(p)|<p^{11/2}$.
Finally, the case of weight $1$ is due to Deligne--Serre, in fact using the
construction for $k\ge2$ and congruences.

\subsection{Origins of Modular Forms}

Modular forms are all pervasive in mathematics, physics, and combinatorics.
We just want to mention the most important constructions:

\begin{itemize}
\item Historically, the first modular forms were probably
  \emph{theta functions} (this dates back to J.~Fourier at the end of the
  18th century in his treatment of the \emph{heat equation}) such as
  $\th(\tau)$ seen above, and more generally theta functions associated to
  \emph{lattices}. These functions can have integral or half-integral weight
  (see below) depending on whether the number of variables which occur
  (equivalently, the dimension of the lattice) is even or odd.
  Later, these theta functions were generalized by introducing
  \emph{spherical polynomials} associated to the lattice.

  For example, the theta function associated to the lattice $\Z^2$ is simply
  $f(\tau)=\sum_{(x,y)\in\Z^2}q^{x^2+y^2}$, which is clearly equal to $\th^2$,
  so belongs to $M_1(\G_0(4),\chi_{-4})$. But we can also consider for instance
  $$f_5(\tau)=\sum_{(x,y)\in\Z^2}(x^4-6x^2y^2+y^4)q^{x^2+y^2}\;,$$
  and show that $f_5\in S_5(\G_0(4),\chi_{-4})$:

  \begin{exercise}\begin{enumerate}
    \item Using the notation and results of Exercise \ref{ex:RC},
    show that $[\th,\th]_2=cf_5$ for a suitable constant $c$, so that
    in particular $f_5\in S_5(\G_0(4),\chi_{-4})$.
  \item Show that the polynomial $P(x,y)=x^4-6x^2y^2+y^4$ is a
    \emph{spherical polynomial}, in other words that $D(P)=0$,
    where $D$ is the Laplace differential operator
    $D=\partial^2/\partial^2x+\partial^2/\partial^2y$.
    \end{enumerate}
  \end{exercise}
\item The second occurrence of modular forms is probably
  \emph{Eisenstein series}, which in fact are the first that we encountered
  in this course. We have only seen the most basic Eisenstein series $G_k$
  (or normalized versions) on the full modular group and a few on $\G_0(4)$,
  but there are very general constructions over any space such as
  $M_k(\G_0(N),\chi)$.
  Their Fourier expansions can easily be explicitly computed and are similar
  to what we have given above. More difficult is the case when $k$ is only
  half-integral, but this can also be done.

  As we have seen, an important generalization of Eisenstein series are
  \emph{Poincar\'e series}, which an also be defined over any space as above.
\item A third important construction of modular forms comes from the
  Dedekind eta function $\eta(\tau)$ defined above. In itself it has a
  complicated \emph{multiplier system}, but if we define an \emph{eta quotient}
  as $F(\tau)=\prod_{m\in I}\eta(m\tau)^{r_m}$ for a certain set $I$ of
  positive integers and exponents $r_m\in\Z$, then it is not difficult to
  write necessary and sufficient conditions for $F$ to belong to some
  $M_k(\G_0(N),\chi)$. The first example that we have met is of course
  the Ramanujan delta function $\Delta(\tau)=\eta(\tau)^{24}$. Other examples
  are for instance $\eta(\tau)\eta(23\tau)\in S_1(\G_0(23),\chi_{-23})$,
  $\eta(\tau)^2\eta(11\tau)^2\in S_2(\G_0(11))$, and
  $\eta(2\tau)^{30}/\eta(\tau)^{12}\in S_9(\G_0(8),\chi_{-4})$.
\item Closely related to eta quotients are $q$-identities involving the
  $q$-Pochhammer symbol $(q)_n$ and generalizing those seen in Exercise
  \ref{expoch}, many of which give modular forms not related to the
  eta function.
\item A much deeper construction comes from algebraic geometry: by the
  modularity theorem of Wiles et al., to any elliptic curve defined over $\Q$
  is associated a modular form in $S_2(\G_0(N))$ which is a normalized
  Hecke eigenform, where $N$ is the so-called \emph{conductor} of the curve.
  For instance the eta quotient of level $11$ just seen above is the
  modular form associated to the isogeny class of the elliptic curve of
  conductor $11$ with equation $y^2+y=x^3-x^2-10x-20$.
\end{itemize}

\section{More General Modular Forms}

In this brief section, we will describe modular forms of a more general
kind than those seen up to now.

\subsection{Modular Forms of Half-Integral Weight}

Coming back again to the function $\th$, the formulas seen above
suggest that $\th$ itself must be considered a modular form, of weight
$1/2$. We have already mentioned that
$$\th^2(\ga(\tau))=\leg{-4}{d}(c\tau+d)\th^2(\tau)\;.$$
But what about $\th$ itself? For this, we must be very careful about
the determination of the square root:

Notation: $z^{1/2}$ will \emph{always} denote the principal determination
of the square root, i.e., such that $-\pi/2<\Arg(z^{1/2})\le\pi/2$. For
instance $(2i)^{1/2}=1+i$, $(-1)^{1/2}=i$. Warning: we do not in general
have $(z_1z_2)^{1/2}=z_1^{1/2}z_2^{1/2}$, but only up to sign.
As a second notation, when $k$ is odd, $z^{k/2}$ will \emph{always}
denote $(z^{1/2})^k$ and \emph{not} $(z^k)^{1/2}$ (for instance
$(2i)^{3/2}=(1+i)^3=-2+2i$, while $((2i)^3)^{1/2}=2-2i$).

Thus, let us try and take the square root of the modularity equation for
$\th^2$:
$$\th(\ga(\tau))=v(\ga,\tau)\leg{-4}{d}^{1/2}(c\tau+d)^{1/2}\;,$$
where $v(\ga,\tau)=\pm1$ and may depend on $\ga$ and $\tau$. A detailed
study of Gauss sums shows that $v(\ga,\tau)=\lgs{-4c}{d}$, the general
Kronecker symbol, so that the modularity equation for $\th$ is,
for any $\ga\in\G_0(4)$:
$$\th(\ga(\tau))=v_{\th}(\ga)(c\tau+d)^{1/2}\th(\tau)\text{\quad with\quad}v_{\th}(\ga)=\leg{c}{d}\leg{-4}{d}^{-1/2}\;.$$
Note that there is something very subtle going on here: this complicated
\emph{theta multiplier system} $v_{\th}(\ga)$ must satisfy a complicated
\emph{cocycle relation} coming from the trivial identity
$\th((\ga_1\ga_2)(\tau))=\th(\ga_1(\ga_2(\tau)))$ which can be shown to be
equivalent to the general \emph{quadratic reciprocity law}.

The following definition is due to G.~Shimura:

\begin{definition} Let $k\in1/2+\Z$. A function $F$ from $\H$ to $\C$
  will be said to be a modular form of (half integral) weight $k$ on
  $\G_0(N)$ with character $\chi$ if for all
  $\ga=\psmm{a}{b}{c}{d}\in\G_0(N)$ we have
  $$F(\ga(\tau))=v_{\th}(\ga)^{2k}\chi(d)(c\tau+d)^kF(\tau)\;,$$
  and if the usual holomorphy and conditions at the cusps are satisfied
  (equivalently if $F^2\in M_{2k}(\G_0(N),\chi^2\chi_{-4})$).
\end{definition}

Note that if $k\in1/2+\Z$ we have $v_{\th}(\ga)^{4k}=\chi_{-4}$, which explains
the extra factor $\chi_{-4}$ in the above definition.

Since $v_{\th}(\ga)$ is defined only for $\ga\in\G_0(4)$ we need
$\G_0(N)\subset\G_0(4)$, in other words $4\mid N$. In addition, by definition
$v_{\th}(\ga)(c\tau+d)^{1/2}=\th(\ga(\tau))/\th(\tau)$ is invariant if we
change $\ga$ into $-\ga$, so if $k\in1/2+\Z$ the same is true of
$v_{\th}(\ga)^{2k}(c\tau+d)^k$, hence it follows that in the above definition
we must have $\chi(-d)=\chi(d)$, i.e., $\chi$ must be an \emph{even} character
($\chi(-1)=1$).

As usual, we denote by $M_k(\G_0(N),\chi)$ and $S_k(\G_0(N),\chi)$ the
spaces of modular and cusp forms. The theory is more difficult than the theory
in integral weight, but is now well developed. We mention a few items:

\begin{enumerate}\item There is an explicit but more complicated
  \emph{dimension formula} due to J.~Oesterl\'e and the author.
\item By a theorem of Serre--Stark, modular forms of weight $1/2$ are
  simply linear combinations of \emph{unary theta functions} generalizing the
  function $\th$ above.
\item One can easily construct Eisenstein series, but the computation
  of their Fourier expansion, due to Shimura and the author, is more
  complicated.
\item As usual, if we can express $\th^m$ solely in terms of Eisenstein series,
  this leads to explicit formulas for $r_m(n)$, the number of representation
  of $n$ as a sum of $m$ squares. Thus, we obtain explicit formulas for
  $r_3(n)$ (due to Gauss), $r_5(n)$ (due to Smith and Minkowski), and
  $r_7(n)$, so if we complement the formulas in integral weight, we have
  explicit formulas for $r_m(n)$ for $1\le m\le 8$ and $m=10$.
\item The deeper part of the theory, which is specific to the half-integral
  weight case, is the existence of \emph{Shimura lifts} from
  $M_k(\G_0(N),\chi)$ to $M_{2k-1}(\G_0(N/2),\chi^2)$, the description of
  the \emph{Kohnen subspace} $S_k^+(\G_0(N),\chi)$ which allows both the
  Shimura lift to go down to level $N/4$, and also to define a suitable
  Atkin--Lehner type new space, and the deep results of Waldspurger, which nicely
  complement the work of Shimura on lifts.
\end{enumerate}

\medskip

We could try to find other types of interesting modularity properties
than those coming from $\th$. For instance, we have seen that the Dedekind
eta function is a modular form of weight $1/2$ (not in Shimura's sense),
and more precisely it satisfies the following modularity equation, now
for any $\ga\in\G$:
$$\eta(\ga(\tau))=v_{\eta}(\ga)(c\tau+d)^{1/2}\eta(\tau)\;,$$
where $v_{\eta}(\ga)$ is a very complicated $24$-th root of unity.
We could of course define $\eta$-modular forms of half-integral weight
$k\in1/2+\Z$ by requiring $F(\ga(\tau))=v_{\eta}(\ga)^{2k}(c\tau+d)^kF(\tau)$,
but it can be shown that this would not lead to any interesting theory
(more precisely the only interesting functions would be \emph{eta-quotients}
$F(\tau)=\prod_m\eta(m\tau)^{r_m}$, which can be studied directly without
any new theory.

Note that there are functional relations between $\eta$ and $\th$:

\begin{proposition} We have
  $$\th(\tau)=\dfrac{\eta^2(\tau+1/2)}{\eta(2\tau+1)}=\dfrac{\eta^5(2\tau)}{\eta^2(\tau)\eta^2(4\tau)}\;.$$\end{proposition}

\begin{exercise}\begin{enumerate}\item Prove these relations in the following
way: first show that the right-hand sides satisfy the same modularity
equations as $\th$ for $T=\psmm{1}{1}{0}{1}$ and $W_4=\psmm{0}{-1}{4}{0}$, so
in particular that they are weakly modular on $\G_0(4)$, and second show that
they are really modular forms, in other words that they are holomorphic on
$\H$ and at the cusps.
\item Using the definition of $\eta$, deduce two \emph{product expansions} for
$\th(\tau)$.\end{enumerate}
\end{exercise}

We could also try to study modular forms of fractional or even real weight
$k$ not integral or half-integral, but this would lead to functions with no
interesting \emph{arithmetical} properties.

In a different direction, we can relax the condition of holomorphy (or
meromorphy) and ask that the functions be eigenfunctions of the
\emph{hyperbolic Laplace operator}
$$\Delta=-y^2\left(\dfrac{\partial^2}{\partial^2 x}+\dfrac{\partial^2}{\partial^2 y}\right)=-4y^2\dfrac{\partial^2}{\partial\tau\partial\ov{\tau}}$$
which can be shown to be invariant under $\G$ (more generally under
$\SL_2(\R)$) together with suitable boundedness conditions. This leads to the
important theory of \emph{Maass forms}.
The case of the eigenvalue $0$ reduces to ordinary modular forms since
$\Delta(F)=0$ is equivalent to $F$ being a linear combination of a
holomorphic and antiholomorphic (i.e., conjugate to a holomorphic) function,
each of which will be modular or conjugate of modular.

The case of the eigenvalue $1/4$ also leads to functions having nice
arithmetical properties, but all other eigenvalues give functions with
(conjecturally) transcendental coefficients, but these functions are useful
in number theory for other reasons which we cannot explain here. Note that a
famous conjecture of Selberg asserts that for \emph{congruence subgroups} there
are no eigenvalues $\la$ with $0<\la<1/4$. For instance, for the full modular
group, the smallest nonzero eigenvalue is $\la=91.1412\cdots$, which is
quite large.

\begin{exercise} Using the fact that $\Delta$ is invariant under $\G$ show
that $\Delta(\Im(\ga(\tau)))=s(1-s)\Im(\ga(\tau))$ and deduce that the
nonholomorphic Eisenstein series $E(s)$ introduced in Definition
\ref{def:nonhol} is an eigenfunction of the hyperbolic Laplace operator with
eigenvalue $s(1-s)$ (note that it does not satisfy the necessary boundedness
conditions, so it is not a Maass form: the functions $E(s)$ with $\Re(s)=1/2$
constitute what is called the \emph{continuous spectrum}, and the Maass forms
the \emph{discrete spectrum} of $\Delta$ acting on $\G\backslash\H$).
\end{exercise}

\subsection{Modular Forms in Several Variables}\label{sec:several}

The last generalization that we want to mention (there are much more!) is to
several variables. The natural idea is to consider holomorphic functions
from $\H^r$ to $\C$, now for some $r>1$, satisfying suitable modularity
properties. If we simply ask that $\ga\in\G$ (or some subgroup) acts
component-wise, we will not obtain anything interesting. The right way to do
it, introduced by Hilbert--Blumenthal, is to consider a \emph{totally real}
number field $K$ of degree $r$, and denote by $\G_K$ the group of matrices
$\ga=\psmm{a}{b}{c}{d}\in\SL_2(\Z_K)$, where $\Z_K$ is the ring of algebraic
integers of $K$ (we could also consider the larger group $\GL_2(\Z_K)$, which
leads to a very similar theory). Such a $\ga$ has $r$ \emph{embeddings}
$\ga_i$ into $\SL_2(\R)$, which we will denote by
$\ga_i=\psmm{a_i}{b_i}{c_i}{d_i}$, and the correct definition is to ask that
$$F(\ga_1(\tau_1),\cdots,\ga_r(\tau_r))=(c_1\tau_1+d_1)^k\cdots(c_r\tau_r+d_r)^kF(\tau_1,\dots,\tau_r)\;.$$
Note that the restriction to totally real number fields is due to the fact
that for $\ga_i$ to preserve the upper-half plane it is necessary that
$\ga_i\in\SL_2(\R)$. Note also that the $\ga_i$ are \emph{not} independent,
they are conjugates of a single $\ga\in\SL_2(\Z_K)$.

A holomorphic function satisfying the above is called a
\emph{Hilbert-Blumenthal} modular form (of \emph{parallel weight $k$}, one
can also consider forms where the exponents for the different embeddings
are not equal), or more simply a Hilbert modular form
(note that there are no ``conditions at infinity'', since one can prove that
they are automatically satisfied unless $K=\Q$).

Since $T=\psmm{1}{1}{0}{1}\in\SL_2(\Z_K)$ is equal to all its conjugates, such
modular forms have Fourier expansions, but using the action of
$\psmm{1}{\al}{0}{1}$ with $\al\in\Z_K$ it is easy to show that these
expansions are of a special type, involving the \emph{codifferent}
${\mathfrak d}^{-1}$ of $K$, which is the fractional ideal of $x\in K$ such
that $\Tr(x\Z_K)\subset\Z$, where $\Tr$ denotes the trace.

One can construct Eisenstein series, here called Hecke--Eisenstein series,
and compute their Fourier expansion. One of the important consequences of this
computation is that it gives an explicit formula for the value $\z_K(1-k)$
of the \emph{Dedekind zeta function} of $K$ at negative integers (hence by
the functional equation of $\z_K$, also at positive even integers), and in
particular it proves that these values are \emph{rational numbers}, a theorem
due to C.-L.~Siegel as an immediate consequence of Theorem \ref{thmsieg}.
An example is as follows:

\begin{proposition} Let $K=\Q(\sqrt{D})$ be a real quadratic field with
  $D$ a fundamental discriminant. Then:
  \begin{enumerate}
  \item We have
    \begin{align*}\z_K(-1)&=\dfrac{1}{60}\sum_{|s|<\sqrt{D}}\sigma_1\left(\dfrac{D-s^2}{4}\right)\;,\\
      \z_K(-3)&=\dfrac{1}{120}\sum_{|s|<\sqrt{D}}\sigma_3\left(\dfrac{D-s^2}{4}\right)\;.\end{align*}
  \item We also have formulas such as
      \begin{align*}
        \sum_{|s|<\sqrt{D}}\sigma_1(D-s^2)&=60\left(9-2\leg{D}{2}\right)\z_K(-1)\;,\\
        \sum_{|s|<\sqrt{D}}\sigma_3(D-s^2)&=120\left(129-8\leg{D}{2}\right)\z_K(-3)\;.\\
      \end{align*}
  \end{enumerate}
\end{proposition}

We can of course reformulate these results in terms of $L$-functions by using
$L(\chi_D,-1)=-12\z_K(-1)$ and $L(\chi_D,-3)=120\z_K(-3)$, where as usual
$\chi_D$ is the quadratic character modulo $D$.

\begin{exercise}
  Using Exercise \ref{exr4r8} and the above formulas, show that the number
  $r_5(D)$ of representations of $D$ as a sum of $5$ squares is given by
  $$r_5(D)=480\left(5-2\leg{D}{2}\right)\z_K(-1)=-40\left(5-2\leg{D}{2}\right)L(\chi_D,-1)\;.$$
\end{exercise}

Note that this formula can be generalized to arbitrary $D$, and is due to
Smith and (much later) to Minkowski. There also exists a similar formula
for $r_7(D)$: when $-D$ (\emph{not} $D$) is a fundamental discriminant
$$r_7(D)=-28\left(41-4\leg{D}{2}\right)L(\chi_{-D},-2)\;.$$

\smallskip

Note also that if we restrict to the \emph{diagonal} $\tau_1=\cdots=\tau_r$,
a Hilbert modular form of (parallel) weight $k$ gives rise to an ordinary
modular form of weight $kr$.

\medskip

We finish this section with some terminology with no explanation: if $K$ is
\emph{not} a totally real number field, one can also define modular forms,
but they will not be defined on products of the upper-half plane $\H$ alone,
but will also involve the \emph{hyperbolic $3$-space} $\H_3$. Such forms
are called \emph{Bianchi} modular forms.

A different generalization, close to the Weierstrass $\wp$-function seen above,
is the theory of \emph{Jacobi forms}, due to M.~Eichler and D.~Zagier. One of
the many interesting aspects of this theory is that it mixes in a nontrivial
way properties of forms of integral weight with forms of half-integral weight.

Finally, we mention \emph{Siegel modular forms}, introduced by C.-L.~Siegel,
which are defined on higher-dimensional \emph{symmetric spaces}, on which the
\emph{symplectic groups} $\Sp_{2n}(\R)$ act. The case $n=1$ gives ordinary
modular forms, and the next simplest, $n=2$, is
closely related to Jacobi forms since the Fourier coefficients of Siegel
modular forms of degree $2$ can be expressed in terms of Jacobi forms.

\section{Some Pari/GP Commands}

There exist three software packages which are able to compute with modular
forms: {\tt magma}, {\tt Sage}, and {\tt Pari/GP} since the spring of 2018.
We give here some basic {\tt Pari/GP} commands with little or no explanation
(which is available by typing {\tt ?} or {\tt ??}): we encourage the reader to
read the tutorial {\tt tutorial-mf} available with the distribution and to
practice with the package, since it is an excellent way to learn about modular
forms. All commands begin with the prefix {\tt mf}, with the exception of
{\tt lfunmf} which more properly belongs to the $L$-function package.

Creation of modular forms: {\tt mfDelta} (Ramanujan Delta), {\tt mfTheta}
(ordinary theta function), {\tt mfEk} (normalized Eisenstein series $E_k$),
more generally {\tt mfeisenstein}, {\tt mffrometaquo} (eta quotients),
{\tt mffromqf} (theta function of lattices with or without spherical
polynomial), {\tt mffromell} (from elliptic curves over $\Q$), etc...

Arithmetic operations: {\tt mfcoefs} (Fourier coefficients at infinity),
{\tt mflinear} (linear combination, so including
addition/subtraction and scalar multiplication), {\tt mfmul}, {\tt mfdiv},
{\tt mfpow} (clear), etc...

Modular operations: {\tt mfbd}, {\tt mftwist}, {\tt mfhecke}, {\tt mfatkin},
{\tt mfderivE2}, {\tt mfbracket}, etc...

Creation of modular form \emph{spaces}: {\tt mfinit}, {\tt mfdim}
(dimension of the space), {\tt mfbasis} (random basis of the space),
{\tt mftobasis} (decomposition of a form on the {\tt mfbasis}),
{\tt mfeigenbasis} (basis of normalized eigenforms).

Searching for modular forms with given Fourier coefficients:

{\tt mfeigensearch}, {\tt mfsearch}.

Expansion of $F|_k\ga$: {\tt mfslashexpansion}.

Numerical functions: {\tt mfeval} (evaluation at a point in $\H$ or at a cusp),
{\tt mfcuspval} (valuation at a cusp), {\tt mfsymboleval} (computation of
integrals over paths in the completed upper-half plane), {\tt mfpetersson}
(Petersson scalar product), {\tt lfunmf} ($L$-function associated to a
modular form), etc...

Note that for now {\tt Pari/GP} is the only package for which these last
functions (beginning with {\tt mfslashexpansion}) are implemented.

\section{Suggestions for further Reading}

The literature on modular forms is vast, so I will only mention
the books which I am familar with and that in my opinion will be very useful
to the reader. Note that the classic book \cite{Shi} is absolutely remarkable,
but may be difficult for a beginning course.

In addition to the recent book \cite{Coh-Str} by F.~Str\"omberg and the author
(which of course I strongly recommend !!!), I also highly recommend the paper
\cite{Zag}, which is essentially a small book. Perhaps the most classical
reference is \cite{Miy}. The more recent book \cite{Dia-Shu} is more
advanced since its ultimate goal is to explain the modularity theorem of
Wiles et al.

\end{document}